\documentclass[11pt,amsfonts]{article}
\usepackage{graphicx}
\usepackage{latexsym}
\usepackage{amssymb}
\usepackage{amsmath}
\usepackage{enumerate}

\usepackage{amsmath}
\usepackage{stmaryrd}
\usepackage{hyperref}
\usepackage{amssymb}
\usepackage{CJK}
\usepackage{color}

\usepackage{amsmath,latexsym,amssymb,amsthm,array,amsfonts}

\usepackage{layout}

\usepackage{tikz}

\newtheorem{prop}{Proposition}
\newtheorem{lemma}{Lemma}

\newtheorem{theorem}{Theorem}

\newtheorem{remark}{Remark}

\newtheorem{condition}{Condition}

\def\real{{\mathord{{\rm I\kern-2.8pt R}}}}        
\def\inte{{\mathord{{\rm I\kern-2.8pt N}}}}

\def\sZZ{{\rm Z\kern-2.8ptem{}Z}}

\def\z{{\mathchoice
		{\sZZ}
		{\sZZ}
		{\rm Z\kern-0.30em{}Z}
		{\rm Z\kern-0.25em{}Z} }}
\def\sQQ{{\kern 0.27em \vrule height1.45ex width0.03em depth0em
		\kern-0.30em \rm Q}}
\def\qu{{\mathchoice
		{\sQQ}
		{\sQQ}
		{\kern 0.225em \vrule height1.05ex width0.025em depth0em \kern-0.25em \rm Q}
		{\kern 0.180em \vrule height0.78ex width0.020em depth0em \kern-0.20em \rm Q}
}}
\def\sCC{{\kern 0.27em \vrule height1.45ex width0.03em depth0em
		\kern-0.30em \rm C}}
\def\complex{{\mathchoice
		{\sCC}
		{\sCC}
		{\kern 0.225em \vrule height1.05ex width0.025em depth0em \kern-0.25em \rm C}
		{\kern 0.180em \vrule height0.78ex width0.020em depth0em \kern-0.20em \rm C}
}}

\def\qed{ \hfill \vrule width.25cm height.25cm depth0cm\smallskip}

\newcommand{\ignore}[1]{}
\newcommand{\CQFD}{\hfill $\Box$}

\usepackage{geometry}
\begin{document}
	
	\renewcommand{\thefootnote}{\fnsymbol{footnote}}
	
	\renewcommand{\thefootnote}{\fnsymbol{footnote}}

	\title{Multidimensional Stein's method for asymptotic independence with invariant measures of diffusion} 
	
	\author{Ciprian A. Tudor (Université de Lille) \\  J\'er\'emy Zurcher (Georgia Institute of Technology)\vspace*{0.2in}
	\vspace*{0.1in} }
	
	\maketitle

    \begin{abstract}
        We derive a multidimensional Stein's method for asymptotic independence in the case of a general target $\mu$ with a density, being invariant measure of a diffusion process. It allows us to give a general bound in Wasserstein distance between the law of a couple $(X, Y)$, where $X$ is a random variable, and $Y$ a random vector and $\mu \otimes \mathrm{Law}(Y)$. We focus in particular in the case where $X$ and $Y$ are differentiable in the Malliavin sense, by being function of a finite number of stochastic Wiener integrals.
    \end{abstract}
	
\section{Introduction}

Stein's method has become a central tool in modern probability theory for obtaining
quantitative bounds in distributional approximations. By characterizing a target law
through an associated differential operator and Stein equation, it allows one to control
distances between probability measures in various metrics such as the Wasserstein,
Kolmogorov, or total variation distances. Originally developed for Gaussian
approximation, the method has since been extended to a wide class of distributions and
multidimensional settings.

A major breakthrough in this direction is the combination of Stein's method with
Malliavin calculus, initiated in the Gaussian framework (see \cite{NP-book} for a
complete exposition) and later generalized to other target distributions. In particular,
a large class of probability distributions to which the Stein--Malliavin calculus can be
applied is the class of the so-called invariant measures of diffusions, which includes
the best known probability laws with density (among others, Gaussian, Gamma, uniform, or
Beta distributions). We refer to \cite{EDVI}, \cite{TuKus}, or \cite{TuKus2} for the
development of the Stein--Malliavin calculus for such target distributions.

More recently, attention has shifted toward multidimensional extensions and asymptotic
independence problems. In this context, the multidimensional Stein--Malliavin framework
developed in \cite{T2, TuZu2} provides quantitative bounds for the distance between the
joint law of a pair $(X, Y)$ and a product measure of the form
$\mathcal{N}(0,1) \otimes \mathbf{P}_Y$, where the first marginal is Gaussian and
independent of the second. We refer to \cite{Pi}, \cite{T2}, \cite{TuZu1},
\cite{TuZu2}, \cite{TuZu3}, \cite{BaZh}, or \cite{SaMa} for some recent results in
this direction. This line of research is closely related to earlier works on asymptotic
independence on Wiener chaos (see \cite{NR}, \cite{NNP}) and highlights the deep
interplay between Malliavin calculus, contraction operators, and independence phenomena.

The aim of the present paper is to contribute to this line of research by developing a
multidimensional Stein's method for asymptotic independence in the case of a general
target distribution with density, going substantially beyond the Gaussian framework of
\cite{T2, TuZu2} and the Gamma framework of \cite{TuZu1}. More precisely, given a
probability measure $\mu$ on $\mathbb{R}$ which is the invariant measure of a diffusion
process and which we assume admits a density, we study the distance between the joint law
$\mathbf{P}_{(X, Y)}$ and the product measure $\mu \otimes \mathbf{P}_{Y}$. Our approach
relies on the construction and analysis of a Stein equation associated with $\mu$, whose
structure is naturally linked to the generator of a diffusion having $\mu$ as invariant
measure, in the spirit of \cite{TuKus} or \cite{EDVI}.

What makes the present contribution new and different from the existing literature is
threefold. First, while \cite{T2} and \cite{TuZu2} treat the case of a Gaussian target
and \cite{TuZu1} addresses the Gamma distribution, we work here with an arbitrary
probability measure $\mu$ that is the invariant measure of a diffusion --- a class that
encompasses all these cases and many more. Second, unlike the Gaussian setting where the
diffusion coefficient is constant and explicit formulas for the Stein solution are
available, our general setting requires a careful analysis of the diffusion coefficient
$a(x)$ near the boundary of the support of $\mu$, leading to new structural conditions
and a new technical framework. Third, the asymptotic independence criterion we derive
(Theorem~\ref{tt1}) is expressed entirely in terms of Malliavin operators applied to $X$
and $Y$, providing a unified and computable criterion that recovers the known results of
\cite{T2, TuZu1, TuZu2} as special cases.

A key difficulty in this setting lies in obtaining sharp bounds on the derivatives of the
solution to the Stein equation. By introducing suitable structural conditions and smooth
approximations of $a$ near the edges of the support, we derive bounds that depend only
on the Lipschitz norm of the test function --- a feature that is essential for the
Stein--Malliavin approach and that was not available for general invariant measures before
the present work. Our main results thus extend the sharp estimates of \cite{TuKus, EDVI}
to the multidimensional and asymptotic independence setting, and the asymptotic
independence bounds of \cite{T2, TuZu1, TuZu2} to a broad class of non-Gaussian target
distributions.

Finally, we illustrate the applicability of our results through three examples involving
functionals of isonormal Gaussian processes: convergence of second-chaos sequences to a
Gamma distribution, asymptotic independence for functionals related to the uniform
distribution on $[0,1]$, and an exponential functional converging to the lognormal
distribution. These examples show how asymptotic independence emerges in genuinely
non-Gaussian settings, complementing the existing results on Wiener chaos and
multidimensional convergence of \cite{NR, NNP, T2, TuZu3}.

The paper is organized as follows. Section~2 introduces the framework and the Stein
operator associated with $\mu$. Section~3 is devoted to solving the Stein equation and
deriving sharp bounds on its solution and derivatives. Section~4 contains the main
theorem on Wasserstein distance and asymptotic independence via Malliavin calculus.
Section~5 presents the applications. Section~6 collects the necessary tools from
Malliavin calculus.

\section{Framework}

Let $(\Omega, \mathcal{F}, \mathbf{P})$ be a probability space. Let $\mu$ be a
probability measure on $\mathbb{R}$ admitting a density $p : (l,u) \to \mathbb{R}_+$,
where $-\infty \leq l < u \leq +\infty$. We assume that $p$ is strictly positive on
$(l,u)$ and admits a finite first moment. We denote by
\[
m := \int_l^u t\,p(t)\,\mathrm{d}t
\]
the mean of $\mu$, and by $F$ its cumulative distribution function.

Following the general philosophy of Stein's method for invariant measures of diffusions
(see, e.g.\ \cite{TuKus} and \cite{EDVI}), we associate to $\mu$ a first-order
differential operator of the form
\[
\mathcal{A}f(x) = \frac{1}{2}a(x)f'(x) - (x-m)f(x),
\]
where the diffusion coefficient $a : (l,u) \to \mathbb{R}_+$ is defined by
\begin{equation}\label{defa}
	a(x) = \frac{2}{p(x)} \int_l^x (m - t)p(t)\,\mathrm{d}t
	= \frac{2}{p(x)} \int_x^u (t - m)p(t)\,\mathrm{d}t.
\end{equation}
By construction, $a(x) > 0$ for all $x \in (l,u)$. The measure $\mu$ is the invariant
measure of the diffusion process $(X_{t}, t\geq 0)$ given by
$\mathrm{d}X_{t} = -(X_{t}-m)\,\mathrm{d}t + \sqrt{a(X_{t})}\,\mathrm{d}W_{t}$,
where $W$ is a Wiener process; see \cite{BSS}.

The operator $\mathcal{A}$ characterizes the measure $\mu$ in the sense that, for every
function $f \in C^{1}(l,u)$,
\begin{equation}\label{22a-1}
	\mathbf{E}[\mathcal{A}f(Z)] = 0 \quad \text{if and only if } Z \sim \mu,
\end{equation}
see Theorem~3.2 in \cite{TuKus2}. This property can be easily extended to the
multidimensional context as follows.

\begin{lemma}
	Let $X \sim \mu$ be a $(l,u)$-valued random variable and let $Y$ be another random
	vector independent of $X$, of dimension $d$. Then for every
	$h \in C^{1}\!\left((l,u) \times \mathbb{R}^d\right)$,
	\[
	\mathbf{E} \left[ \frac{1}{2}a(X)\frac{\partial h}{\partial x}(X,Y)
	-(X-m)\,h(X,Y)\right] = 0.
	\]
\end{lemma}

\noindent{\bf Proof:} By \eqref{22a-1}, for every $y \in \mathbb{R}^d$,
\[
\mathbf{E}\!\left[ \frac{1}{2}a(X)\frac{\partial h}{\partial x}(X,y)
-(X-m)\,h(X,y)\right] = 0,
\]
or equivalently
\[
\int_{l}^{u} \left[ \frac{1}{2}a(x)\frac{\partial h}{\partial x}(x,y)
-(x-m)\,h(x,y)\right] \mathbf{P}_{X}(\mathrm{d}x) = 0,
\]
where $\mathbf{P}_{X}$ is the probability distribution of $X$. By integrating with
respect to $\mathbf{P}_{Y}$ over $\mathbb{R}^d$ and using the independence of $X$
and $Y$,
\begin{eqnarray*}
	&&\int_{\mathbb{R}^d}\int_{l}^{u}
	\left[\frac{1}{2}a(x)\frac{\partial h}{\partial x}(x,y)
	-(x-m)\,h(x,y)\right]\mathbf{P}_{X}(\mathrm{d}x)\,\mathbf{P}_{Y}(\mathrm{d}y)\\
	&=&\int_{(l,u)\times\mathbb{R}^d}
	\left[\frac{1}{2}a(x)\frac{\partial h}{\partial x}(x,y)
	-(x-m)\,h(x,y)\right]\mathbf{P}_{X}\otimes\mathbf{P}_{Y}(\mathrm{d}x,\mathrm{d}y)\\
	&=&\int_{(l,u)\times\mathbb{R}^d}
	\left[\frac{1}{2}a(x)\frac{\partial h}{\partial x}(x,y)
	-(x-m)\,h(x,y)\right]\mathbf{P}_{(X,Y)}(\mathrm{d}x,\mathrm{d}y) = 0.
\end{eqnarray*}
The last equality is equivalent to the conclusion of the lemma. \qed

Let $X$ be a real-valued random variable with values in $(l,u)$, and let $Y$ be a
random vector in $\mathbb{R}^d$, both defined on the probability space
$(\Omega,\mathcal{F},\mathbf{P})$. The main objective of this paper is to quantify the
distance between the joint distribution $\mathbf{P}_{(X,Y)}$ and the product measure
$\mu \otimes \mathbf{P}_Y$, that is, to measure how close $X$ is to $\mu$ and how close
$(X,Y)$ is to being independent.

To this end, we adopt a multidimensional Stein approach and consider test functions
$h : (l,u) \times \mathbb{R}^d \to \mathbb{R}$ belonging to suitable smoothness
classes. The associated Stein equation takes the form
\[
\frac{1}{2}a(x)\,\partial_x f(x,y) - (x-m)f(x,y)
= h(x,y) - \mathbf{E}[h(Z,y)],
\]
where $Z \sim \mu$ is independent of $Y$. The analysis of this equation, and in
particular the derivation of uniform bounds on the solution $f$ and its derivatives,
plays a central role in obtaining quantitative estimates.

In order to derive such bounds, special attention must be paid to the behavior of the
coefficient $a(x)$ near the boundaries $l$ and $u$, especially when these are infinite.
As in \cite{EDVI}, we will impose structural conditions ensuring that $a$ does not
vanish and can be approximated by a smoother function near the edges of the support.
These assumptions are crucial to control the so-called Stein factors and to obtain
bounds depending only on the derivatives of the test function.

Finally, in order to connect this framework with Malliavin calculus, we will assume that
the random variables $X$ and $Y$ belong to suitable Sobolev spaces on the Wiener space
(typically $\mathbb{D}^{1,2}$ or $\mathbb{D}^{1,4}$). This allows us to express the
error terms in terms of Malliavin operators such as the derivative $D$ and the
pseudo-inverse of the Ornstein--Uhlenbeck generator $L^{-1}$, in the spirit of
\cite{TuKus}.

This framework sets the stage for the analysis of the Stein equation in the next section
and for the derivation of quantitative bounds for asymptotic independence.

\section{Stein's equation}

Let us consider Stein's equation:

\begin{equation}\label{st1}
	\frac{1}{2} a(x) \frac{\partial f_h}{\partial x} (x, y) - (x-m) f_h (x, y)
	= h(x, y) - \mathbf{E} [h(Z, y)],
\end{equation}

\noindent where $h : (l, u) \times \mathbb{R}^d \longrightarrow \mathbb{R}$ is a test
function whose regularity is specified in Proposition~\ref{propINGf}, and $Z$ follows
$\mu$. This section is devoted to solving \eqref{st1} and to analyzing the properties
of its solution. To do this, we will assume some conditions, which we write in two
parts for clarity. The first concerns the behavior of $a$ at the edges of $(l, u)$
when they are infinite.

\begin{condition}\label{cdn1}
	We assume that $a$ satisfies the following lower bounds, when $l = -\infty$ or
	$u = +\infty$:
	\begin{enumerate}
		\item If $l = -\infty$, $\liminf_{x \to -\infty} a(x) > 0$;
		\item If $u = +\infty$, $\liminf_{x \to +\infty} a(x) > 0$.
	\end{enumerate}
\end{condition}

In other words, when the interval contains an infinite bound, we impose that $a(x)$
does not vanish as $x$ goes to that bound. The second condition concerns approximating
$a$, which may have limited regularity, by a smoother function near the edges. These
conditions were inspired by \cite{EDVI}.

\begin{condition}\label{cdn2}
	We assume that there exists a differentiable function
	$\tilde{a} : (l, u) \longrightarrow \mathbb{R}$ such that:
	\begin{enumerate}
		\item The limits of $\tilde{a}$ at $l$ and $u$ both exist in $\bar{\mathbb{R}}$.
		\item The limits of $\tilde{a}'$ at $l$ and $u$ both exist in $\bar{\mathbb{R}}$
		when $l$ or $u$ is infinite. If $l > -\infty$, we assume
		$\liminf_{x \to l} \tilde{a}'(x) \geqslant 0$, and if $u < \infty$,
		$\liminf_{x \to u} \tilde{a}'(x) \leqslant 0$.
		\item $0 < \liminf_{x \to l} \dfrac{a(x)}{\tilde{a}(x)}
		\leqslant \limsup_{x \to l} \dfrac{a(x)}{\tilde{a}(x)} < \infty$;
		\item $0 < \liminf_{x \to u} \dfrac{a(x)}{\tilde{a}(x)}
		\leqslant \limsup_{x \to u} \dfrac{a(x)}{\tilde{a}(x)} < \infty$.
	\end{enumerate}
\end{condition}

Having made these remarks, let us comment on the above assumptions.

\begin{remark}\label{rqcst}
	\begin{enumerate}
		\item Point~3 can be rephrased as follows: there exists a neighborhood of $l$ and
		constants $c_1, C_1 > 0$ such that $c_1 \tilde{a} \leqslant a \leqslant C_1 \tilde{a}$
		in that neighborhood. The same remark applies in the neighborhood of $u$.
		
		\item At first sight, one could have used the same conditions as in \cite{TuKus},
		Proposition~3, which are simpler. More precisely, it would suffice to suppose that
		$\tilde{a}(x) = x - l$ near $l$ and $\tilde{a}(x) = u - x$ near $u$, when the edges
		are finite. However, as shown in the next proposition, we seek an estimate of the
		solution of Stein's equation \eqref{st1} of the form
		\[ \left\| \frac{\partial f_h}{\partial x} \right\|_{\infty}
		\leqslant C \left\| \frac{\partial h}{\partial x} \right\|_{\infty}. \]
		The conditions of \cite{TuKus} only yield an estimate (see Proposition~3 there) of
		the form
		\[ \left\| \frac{\partial f_h}{\partial x} \right\|_{\infty}
		\leqslant C \left( \| h \|_{\infty} + \left\| \frac{\partial h}{\partial x}
		\right\|_{\infty} \right). \]
		Condition~\ref{cdn2} comes from \cite{EDVI} (Condition~B). The idea there is to
		introduce a smooth approximation of $a$ near the edges, which allows one to bound the
		function $S$ defined in \eqref{S(x)1}. In practice, as the authors note before
		Assumption~B following the proof of Theorem~9, $\tilde{a}$ coincides with $a$ itself.
		
		\item When $a$ belongs to $C^{1}(l,u)$, the function $\tilde{a}$ in Condition~\ref{cdn2}
		coincides with $a$ (see \cite{EDVI}, page~199).
	\end{enumerate}
\end{remark}

With these remarks in hand, we can state the main result of this section. For
$f : (l, u) \times \mathbb{R}^d \longrightarrow \mathbb{R}$ a bounded function, we
denote
\[ \| f \|_{\infty} := \sup_{(x, y) \in (l, u) \times \mathbb{R}^d}
\big| f(x, y) \big|. \]

\begin{prop}\label{propINGf}
	Assume that both Conditions~\ref{cdn1} and \ref{cdn2} are satisfied. Let
	$h : (l, u) \times \mathbb{R}^d \longrightarrow \mathbb{R}$ be a bounded $C^1$
	function with bounded partial derivatives. There exists a solution to \eqref{st1}
	which is bounded on $(l, u) \times \mathbb{R}^d$. It is given for every
	$y \in \mathbb{R}^d$ and $x \in (l, u)$ by:
	\begin{eqnarray}\label{expsol}
		f_h (x, y) &=& \frac{2}{a(x) p(x)} \int_{l}^{x}
		\bigl(h(t, y) - \mathbf{E}[h(Z, y)]\bigr)\, p(t)\,\mathrm{d}t \\
		&=& \frac{2(1 - F(x))}{a(x) p(x)} \int_{l}^x \frac{\partial h}{\partial x}(w, y)
		F(w)\,\mathrm{d}w + \frac{2 F(x)}{a(x) p(x)} \int_x^{u}
		\frac{\partial h}{\partial x}(w, y)(1 - F(w))\,\mathrm{d}w,
		\label{exp2sol}
	\end{eqnarray}
	where $F$ is the cumulative distribution function of $\mu$ and $Z \sim \mu$.
	Moreover, the following inequalities hold:
	\begin{equation}\label{INGf}
		\sup_{(x, y) \in (l, u) \times \mathbb{R}^d} \big| f_h(x, y) \big|
		\leqslant \left\| \frac{\partial h}{\partial x} \right\|_{\infty},
		\qquad
		\max_{1 \leqslant j \leqslant d}\,\sup_{(x, y) \in (l, u) \times \mathbb{R}^d}
		\left| \frac{\partial f_h}{\partial y_j}(x, y) \right|
		\leqslant \frac{2}{a(\mu_m)\,p(\mu_m)}
		\left\| \frac{\partial h}{\partial y_j} \right\|_{\infty},
	\end{equation}
	where $\mu_m$ is the median of $\mu$, and there exists $C > 0$ (equal to
	$\| S \|_{\infty}$ with $S$ defined in Lemma~\ref{Lem:IngS}) such that
	\begin{equation}\label{INGdf}
		\sup_{(x, y) \in (l, u) \times \mathbb{R}^d}
		\left| \frac{\partial f_h}{\partial x}(x, y) \right|
		\leqslant C \left\| \frac{\partial h}{\partial x} \right\|_{\infty}.
	\end{equation}
\end{prop}

The key technical ingredient in the proof is the following lemma.

\begin{lemma}\label{Lem:IngS}
	Define, for all $x \in (l, u)$,
	\begin{equation}\label{S(x)1}
		S(x) := \frac{8 \left( \int_{l}^x F(w)\,\mathrm{d}w \right)
			\left( \int_{x}^{u} (1 - F(w))\,\mathrm{d}w \right)}{a(x)^2\,p(x)}.
	\end{equation}
	Then $S$ is bounded on $(l, u)$.
\end{lemma}

\noindent{\bf Proof:} The function $S$ is well-defined on $(l, u)$ since, by
\eqref{defa}, $a(x) > 0$ for all $x \in (l,u)$, and $p$ is strictly positive on
$(l,u)$ by assumption.

To prove that $S$ is bounded, we note that $S$ is continuous and positive on $(l, u)$,
so it suffices to show that $S$ has finite upper limits at the edges of $(l, u)$. We
follow the same type of computations as in \cite{EDVI}. By a Fubini argument, we have
\begin{eqnarray}
	\int_{l}^x F(w)\,\mathrm{d}w
	&=& x F(x) - \int_{l}^x t\,p(t)\,\mathrm{d}t
	= (x-m) F(x) + \frac{1}{2} a(x) p(x); \label{fubF}\\
	\int_x^{u} (1-F(w))\,\mathrm{d}w
	&=& \int_x^{u} t\,p(t)\,\mathrm{d}t - x(1 - F(x)). \nonumber
\end{eqnarray}
Consequently, whether $l$ and $u$ are finite or not (with the convention
$\frac{1}{\infty} = 0$), we have:
\[ \int_{l}^x F(w)\,\mathrm{d}w \underset{x \to u}{\sim}
\left(1 - \frac{m}{u}\right) x
\quad \text{and} \quad
\int_x^{u} (1 - F(w))\,\mathrm{d}w \underset{x \to l}{\sim}
\left(1 + \frac{m}{l}\right)(-x). \]

\noindent We prove that $\limsup_{x \to l} S(x) < \infty$; the case of the limit at
$u$ is analogous. We have the following equivalent for $S$ near $l$:
\[ S(x) \underset{x \to l}{\sim} \tilde{S}(x)
:= \left(1 + \frac{m}{l}\right) \frac{-x}{a(x)}
\cdot \frac{\int_{l}^x F(w)\,\mathrm{d}w}{a(x)\,p(x)}. \]

\noindent$\bullet$ \textbf{Case $l > -\infty$.} By applying l'H\^opital's rule (more
precisely the version with inferior and superior limits given in \cite{TAY}):
\begin{eqnarray*}
	\limsup_{x \to l} \tilde{S}(x)
	&\leqslant& (m - l) \left[\limsup_{x \to l} \frac{x-l}{a(x)}\right]
	\cdot \limsup_{x \to l}
	\frac{\int_{l}^x F(w)\,\mathrm{d}w}{(x-l)\cdot a(x)\,p(x)} \\
	&\overset{(\hat{H})}{\leqslant}& C_1 (m - l) \limsup_{x \to l}
	\frac{F(x)}{p(x)\,a(x) - 2(x-l)(x-m)\,p(x)} \\
	&\leqslant& C_1 (m - l)
	\frac{1}{\liminf_{x \to l}\!\left(1 - 2(x-m)\tfrac{x-l}{a(x)}\right)}
	\limsup_{x \to l} \frac{F(x)}{a(x)\,p(x)}.
\end{eqnarray*}
The constant $C_1$ is as in Remark~\ref{rqcst}. The symbol
$\overset{(\hat{H})}{\leqslant}$ denotes the use of l'H\^opital's rule. The inferior
limit is strictly positive since
\[ \liminf_{x \to l} \left(1 - 2(x-m)\frac{x-l}{a(x)}\right)
\geqslant 1 + 2c_1(m - l) > 0, \]
where $c_1 > 0$ is as in Remark~\ref{rqcst}. Applying l'H\^opital's rule once more,
\begin{eqnarray*}
	\limsup_{x \to l} \tilde{S}(x)
	&\overset{(\hat{H})}{\leqslant}&
	\frac{C_1(m-l)}{1 + 2c_1(m-l)} \limsup_{x \to l}
	\frac{p(x)}{-2(x-m)\,p(x)}
	= \frac{C_1}{2 + 4c_1(m-l)} < \infty.
\end{eqnarray*}

\noindent$\bullet$ \textbf{Case $l = -\infty$.} The term $(-x)$ requires stronger
conditions on $a$. Under our assumptions, the limit
$L := \lim_{x \to -\infty} \frac{\tilde{a}(x)}{|x|}$ exists in $\bar{\mathbb{R}}$
and is non-negative. Indeed, if $\tilde{a}(x)$ converges to a finite limit as
$x \to -\infty$, then $L = 0$; otherwise, by l'H\^opital's rule,
$L = -\lim_{x \to -\infty} \tilde{a}'(x) \in \bar{\mathbb{R}}$. We consider two
sub-cases.

\indent\indent$\rhd$ \textbf{Sub-case $L > 0$.} Using the bound
\[ \int_{-\infty}^x F(w)\,\mathrm{d}w
= (x-m)F(x) + \frac{1}{2}a(x)p(x) \leqslant \frac{a(x)p(x)}{2}, \]
we obtain
\begin{eqnarray*}
	\limsup_{x \to -\infty} \tilde{S}(x)
	&\leqslant& \left(\limsup_{x \to -\infty} \frac{\tilde{a}(x)}{a(x)}\right)
	\cdot \limsup_{x \to -\infty}
	\frac{|x|\,p(x)\,a(x)}{2\,\tilde{a}(x)\,a(x)\,p(x)}
	= \frac{C_1}{2L} < \infty.
\end{eqnarray*}

\indent\indent$\rhd$ \textbf{Sub-case $L = 0$.} Using instead that
\[ \int_{-\infty}^x F(w)\,\mathrm{d}w
= (x-m)F(x) + \frac{1}{2}a(x)p(x) \geqslant 0, \]
we deduce $F(x) \leqslant \frac{a(x)p(x)}{2(m-x)}$, and hence
\begin{eqnarray*}
	\limsup_{x \to -\infty} \tilde{S}(x)
	&\leqslant& \limsup_{x \to -\infty}
	\frac{\int_{-\infty}^x \frac{|x|}{m+|w|}\,a(w)\,p(w)\,\mathrm{d}w}
	{2\,a(x)\cdot a(x)\,p(x)} \\
	&\leqslant& \frac{1}{\liminf_{x \to -\infty} a(x)}
	\limsup_{x \to -\infty}
	\frac{\int_{-\infty}^x a(w)\,p(w)\,\mathrm{d}w}{a(x)\,p(x)} \\
	&\overset{(\hat{H})}{\leqslant}&
	\frac{1}{\liminf_{x \to -\infty} a(x)}
	\limsup_{x \to -\infty} \frac{a(x)}{-2(x-m)}
	\leqslant \frac{C_1 L}{2\liminf_{x \to -\infty} a(x)} = 0,
\end{eqnarray*}
which is finite since $\liminf_{x \to -\infty} a(x) > 0$ by Condition~\ref{cdn1}.

In all cases we obtain $\limsup_{x \to l} \tilde{S}(x) < \infty$, and hence
$\limsup_{x \to l} S(x) < \infty$. \qed

\vspace{1em}

\noindent\textbf{Proof of Proposition~\ref{propINGf}.}
By using \eqref{defa}, one verifies that $f_h$ given by \eqref{expsol} is indeed a
solution of \eqref{st1} for $x \in (l, u)$. The equivalence of \eqref{expsol} and
\eqref{exp2sol} follows from a Fubini argument: writing for $t \in (l, u)$ and
$y \in \mathbb{R}^d$,
\begin{equation}\label{eq:tilde_h}
	\begin{split}
		\tilde{h}(t, y) := h(t, y) - \mathbf{E}[h(Z, y)]
		&= \mathbf{E}\!\left[\int_{Z}^t \frac{\partial h}{\partial x}(w, y)\,\mathrm{d}w
		\right]
		= \int_{l}^u \!\left(\int_{z}^t \frac{\partial h}{\partial x}(w, y)\,\mathrm{d}w
		\right) p(z)\,\mathrm{d}z \\
		&= \int_l^t \frac{\partial h}{\partial x}(w, y)\,\mathbf{P}(Z \leqslant w)\,\mathrm{d}w
		- \int_t^u \frac{\partial h}{\partial x}(w, y)\,\mathbf{P}(Z \geqslant w)\,\mathrm{d}w,
	\end{split}
\end{equation}
which yields
\begin{equation}\label{eq:Exp_Tilde_h}
	\int_{l}^x \tilde{h}(t, y)\,p(t)\,\mathrm{d}t
	= \mathbf{P}(Z \geqslant x) \int_{l}^x \frac{\partial h}{\partial x}(w, y)\,
	\mathbf{P}(Z \leqslant w)\,\mathrm{d}w
	+ \mathbf{P}(Z \leqslant x) \int_{x}^u \frac{\partial h}{\partial x}(w, y)\,
	\mathbf{P}(Z \geqslant w)\,\mathrm{d}w.
\end{equation}
This proves \eqref{exp2sol}.

\noindent\textit{Proof of \eqref{INGf}, first inequality.}
Let $x \in (l, u)$. By \eqref{exp2sol},
\[ \begin{split}
	\frac{a(x) p(x)}{2}\,|f_h(x, y)|
	&\leqslant \left[\mathbf{P}(Z \geqslant x) \int_{l}^x \mathbf{P}(Z \leqslant w)\,
	\mathrm{d}w
	+ \mathbf{P}(Z \leqslant x) \int_x^{u} \mathbf{P}(Z \geqslant w)\,\mathrm{d}w
	\right] \left\|\frac{\partial h}{\partial x}\right\|_{\infty} \\
	&= \left[- \mathbf{P}(Z \geqslant x) \int_{l}^x t\,p(t)\,\mathrm{d}t
	+ \mathbf{P}(Z \leqslant x) \int_{x}^{u} t\,p(t)\,\mathrm{d}t \right]
	\left\|\frac{\partial h}{\partial x}\right\|_{\infty}.
\end{split} \]
Using $m = \int_l^u t\,p(t)\,\mathrm{d}t$ and \eqref{defa}, this gives
\[ |f_h(x, y)|
\leqslant \frac{2}{a(x)p(x)}\left(\mathbf{P}(Z \geqslant x)\int_l^x (m-t)p(t)\,\mathrm{d}t
+ \mathbf{P}(Z \leqslant x)\int_x^u (t-m)p(t)\,\mathrm{d}t\right)
\left\|\frac{\partial h}{\partial x}\right\|_{\infty}
= \left\|\frac{\partial h}{\partial x}\right\|_{\infty}. \]

\noindent\textit{Proof of \eqref{INGf}, second inequality.}
Since $\partial_{y_j} f_h = f_{\partial_{y_j} h}$, it suffices to show that for every
bounded measurable $\phi : (l, u) \times \mathbb{R}^d \to \mathbb{R}$, one has
$\| f_{\phi} \|_{\infty} \leqslant C\|\phi\|_{\infty}$. From \eqref{expsol} and the
complementary formula, one obtains for every $x \in (l,u)$ and $y \in \mathbb{R}^d$:
\[ \big| f_{\phi}(x, y) \big|
\leqslant 4\min\!\left\{\frac{\int_{l}^x p(t)\,\mathrm{d}t}{a(x)\,p(x)},\,
\frac{\int_{x}^u p(t)\,\mathrm{d}t}{a(x)\,p(x)}\right\}
\|\phi\|_{\infty}
:= 4\min\{A(x), B(x)\}\,\|\phi\|_{\infty}. \]
One checks that $A$ is non-decreasing and $B$ is non-increasing on $(l,u)$. Indeed,
differentiating $A$ and using \eqref{defa},
\[ A'(x) = \frac{1}{a(x)p(x)}\left(p(x)
- \frac{2(m-x)p(x)}{a(x)p(x)}\int_l^x p(t)\,\mathrm{d}t\right)
= \frac{1}{a(x)p(x)}\left(p(x) - p(x)\right) = 0, \]
so $A' \geqslant 0$. By l'H\^opital's rule, $A(x) \to 0$ as $x \to l$ (or
$\frac{1}{2(m-l)}$ if $l > -\infty$) and $A(x) \to +\infty$ as $x \to u$; symmetrically
$B(x) \to +\infty$ as $x \to l$ and $B(x)$ has a finite positive limit at $u$. Hence
$\min\{A(x), B(x)\}$ attains its maximum at the unique point where $A(x) = B(x)$, i.e.\
where $\int_l^x p(t)\,\mathrm{d}t = \int_x^u p(t)\,\mathrm{d}t$, which is $x = \mu_m$,
the median of $\mu$. Therefore
\[ \big| f_{\phi}(x, y) \big|
\leqslant \frac{2}{a(\mu_m)\,p(\mu_m)}\,\|\phi\|_{\infty}, \]
which proves the second inequality in \eqref{INGf}.

\noindent\textit{Proof of \eqref{INGdf}.}
Differentiating \eqref{expsol} with respect to $x$ and using \eqref{defa} and
\eqref{st1}, one derives the identity, for every $x \in (l,u)$ and
$y \in \mathbb{R}^d$,
\[ \frac{1}{4} a(x)^2 p(x) \frac{\partial f_h}{\partial x}(x, y)
= (m-x)\int_l^x \tilde{h}(t, y)\,p(t)\,\mathrm{d}t
+ \frac{a(x)p(x)}{2}\,\tilde{h}(x, y). \]
Substituting \eqref{eq:tilde_h} and \eqref{eq:Exp_Tilde_h}, and then applying the
integration-by-parts identities
\[ \frac{a(x)p(x)}{2}
= (m-x)\mathbf{P}(Z \leqslant x) + \int_l^x \mathbf{P}(Z \leqslant z)\,\mathrm{d}z
= -(m-x)\mathbf{P}(Z \geqslant x) + \int_x^u \mathbf{P}(Z \geqslant z)\,\mathrm{d}z, \]
one obtains
\begin{eqnarray*}
&&	 \frac{1}{4}a(x)^2 p(x)\frac{\partial f_h}{\partial x}(x, y)\\
&&
= \left(\int_x^u \mathbf{P}(Z \geqslant z)\,\mathrm{d}z\right)
\int_l^x \frac{\partial h}{\partial x}(z,y)\,\mathbf{P}(Z \leqslant z)\,\mathrm{d}z
- \left(\int_l^x \mathbf{P}(Z \leqslant z)\,\mathrm{d}z\right)
\int_x^u \frac{\partial h}{\partial x}(z,y)\,\mathbf{P}(Z \geqslant z)\,\mathrm{d}z. 
\end{eqnarray*}
Since $a(x)^2 p(x) > 0$, taking absolute values yields
\[ \left|\frac{\partial f_h}{\partial x}(x, y)\right|
\leqslant \frac{8}{a(x)^2 p(x)}
\left(\int_x^{u} \mathbf{P}(Z \geqslant z)\,\mathrm{d}z\right)
\left(\int_{l}^{x} \mathbf{P}(Z \leqslant z)\,\mathrm{d}z\right)
\left\|\frac{\partial h}{\partial x}\right\|_{\infty}
= S(x)\left\|\frac{\partial h}{\partial x}\right\|_{\infty}, \]
and since $S$ is bounded by Lemma~\ref{Lem:IngS}, this proves \eqref{INGdf}. \CQFD

\medskip
Let us make some comments around the results and proofs in this section. 
\begin{remark}
	
	\begin{enumerate}
		\item 

{\rm

The estimates obtained in Proposition~\ref{propINGf} play a central role in the sequel.
They provide uniform control of the solution to the Stein equation and of its partial
derivatives, which is essential for deriving quantitative bounds in Wasserstein distance. }

\item {\rm A notable feature of the general setting considered here is the presence of the
diffusion coefficient $a(x)$, which reflects the geometry of the target distribution
$\mu$. In contrast with the Gaussian case, where the corresponding coefficient is
constant, the behavior of $a(x)$ near the boundary of the support has a direct impact
on the regularity of the solution $f_h$. This explains the need for
Conditions~\ref{cdn1} and \ref{cdn2}, which ensure that the Stein factors remain
bounded.
}

\item{\rm 

In particular, the quantity $S(x)$ introduced in Lemma~\ref{Lem:IngS} encapsulates the
main analytical difficulty. Its boundedness allows us to control the derivative
$\partial_x f_h$ uniformly, which is crucial for applying Malliavin integration by
parts in the next section. This phenomenon already appears in the non-Gaussian framework
developed in \cite{EDVI}, and becomes unavoidable in the present general setting. }

\item {\rm Another important aspect is that the bounds in \eqref{INGf} and \eqref{INGdf} depend
only on the Lipschitz norm of the test function $h$. This makes them particularly
well-suited for Wasserstein-type distances, where such norms naturally arise. As a
consequence, the Stein equation provides a flexible tool for quantifying both marginal
convergence and asymptotic independence.
}

\item {\rm Finally, we emphasize that the representation formulas \eqref{expsol} and \eqref{exp2sol}
highlight the intrinsic probabilistic structure of the solution $f_h$, which can be
interpreted as an averaged integral of the derivative of the test function against the
distribution function of $\mu$. This representation will be instrumental in the next
section, where it will be combined with Malliavin calculus techniques to derive
quantitative bounds for the distance between $\mathbf{P}_{(X,Y)}$ and
$\mu \otimes \mathbf{P}_Y$.
}
\end{enumerate}
\end{remark}

\section{Estimation of Wasserstein distance via Malliavin calculus}

In this section, we exploit the bounds obtained in Proposition~\ref{propINGf} to derive
quantitative estimates for the distance between the joint law of $(X,Y)$ and the product
measure $\mu \otimes \mathbf{P}_Y$. Our goal is to provide explicit upper bounds in
Wasserstein distance, expressed in terms of Malliavin-type quantities.

The strategy follows the now classical Stein--Malliavin approach: we combine the Stein
equation associated with $\mu$ with integration by parts on the Wiener space. This
allows us to rewrite the difference
\[
\mathbf{E}[h(X,Y)] - \mathbf{E}[h(Z,Y)]
\]
in terms of Malliavin operators applied to $X$ and $Y$, where $Z \sim \mu$ is
independent of $Y$.

A key role is played by the estimates on the derivatives of the solution to the Stein
equation established in Section~3. These bounds make it possible to control the error
terms uniformly over the class of Lipschitz test functions, leading to quantitative
bounds in Wasserstein distance.

The resulting inequalities extend the Gaussian framework developed in the literature to
the case of general invariant measures of diffusions, and provide a natural criterion
for asymptotic independence in terms of Malliavin contractions.

We define the Wasserstein distance between two probability measures $\nu_1$ and $\nu_2$
on $(l,u) \times \mathbb{R}^d$ by
\begin{equation*}
	d_{\mathrm{W}}(\nu_1, \nu_2)
	:= \sup_{h \in \mathrm{Lip}(1)} \left| \int h\,\mathrm{d}\nu_1
	- \int h\,\mathrm{d}\nu_2 \right|,
\end{equation*}
where $\mathrm{Lip}(1)$ is the set of all differentiable functions
$h : (l,u) \times \mathbb{R}^d \longrightarrow \mathbb{R}$ such that
$\| h \|_{\mathrm{Lip}} \leqslant 1$, with
\[
\| h \|_{\mathrm{Lip}}
= \sup_{\substack{((x,y),(x',y')) \in \left((l,u)\times\mathbb{R}^d\right)^2 \\
		(x,y) \neq (x',y')}}
\frac{|h(x,y) - h(x',y')|}{|(x,y)-(x',y')|}.
\]

Let $(\mathcal{H}, \langle \cdot, \cdot \rangle)$ be a real separable Hilbert space on
which we define the usual Malliavin operators (see Section~\ref{sec:tools}). We state
and prove the main result of this section.

\begin{theorem}\label{tt1}
	Assume Conditions~\ref{cdn1} and \ref{cdn2}. Let $X$ be a $(l,u)$-valued real random
	variable and $Y = (Y_1, \ldots, Y_d)$ a real random vector such that both $X$ and $Y$
	belong to $\mathbb{D}^{1,2}$ and $\mathbf{E}[X] = m$. Then
	\begin{eqnarray}\label{dist}
		&&d_{\mathrm{W}}\!\left(\mathbf{P}_{(X,Y)},\, \mu \otimes \mathbf{P}_Y\right)\\
	&&	\leqslant \|S\|_{\infty}\,
		\mathbf{E}\!\left[\left|\frac{1}{2}a(X)
		- \left\langle \mathrm{D}(-L)^{-1}X,\, \mathrm{D}X\right\rangle\right|\right]
		+ \frac{2}{a(\mu_m)\,p(\mu_m)}
		\sum_{j=1}^d \mathbf{E}\!\left[\left|
		\left\langle \mathrm{D}(-L)^{-1}X,\, \mathrm{D}Y_j\right\rangle\right|\right],\nonumber
	\end{eqnarray}
	where $S$ is defined in Lemma~\ref{Lem:IngS}. If moreover $X$ and $Y$ belong to
	$\mathbb{D}^{1,4}$, then
	\begin{multline}\label{dist2}
		d_{\mathrm{W}}\!\left(\mathbf{P}_{(X,Y)},\, \mu \otimes \mathbf{P}_Y\right) \\
		\leqslant \|S\|_{\infty}\,
		\mathbf{E}\!\left[\left(\frac{1}{2}a(X)
		- \left\langle \mathrm{D}(-L)^{-1}X,\,\mathrm{D}X\right\rangle\right)^2
		\right]^{1/2}
		+ \frac{2}{a(\mu_m)\,p(\mu_m)}
		\sum_{j=1}^d \mathbf{E}\!\left[
		\left\langle \mathrm{D}(-L)^{-1}X,\,\mathrm{D}Y_j\right\rangle^2
		\right]^{1/2}.
	\end{multline}
\end{theorem}

\noindent{\bf Proof:}
Let $h : (l,u) \times \mathbb{R}^d \longrightarrow \mathbb{R}$ be a $C^1$ function
such that $\|h\|_{\mathrm{Lip}} \leqslant 1$, and let $Z \sim \mu$ be independent of
$Y$. By Stein's equation \eqref{st1}, we have
\begin{equation*}
	\mathbf{E}[h(X,Y)] - \mathbf{E}[h(Z,Y)]
	= \mathbf{E}\!\left[\frac{1}{2}a(X)\frac{\partial f_h}{\partial x}(X,Y)
	- (X-m)\,f_h(X,Y)\right].
\end{equation*}
Using the identity $\delta\bigl(\mathrm{D}(-L)^{-1}X\bigr) = X - m$
(see Section~\ref{sec:tools}), we obtain
\begin{eqnarray*}
	&&\mathbf{E}[h(X,Y)] - \mathbf{E}[h(Z,Y)] \\
	&=& \frac{1}{2}\mathbf{E}\!\left[a(X)\frac{\partial f_h}{\partial x}(X,Y)\right]
	- \mathbf{E}\!\left[\delta\bigl(\mathrm{D}(-L)^{-1}X\bigr)\cdot f_h(X,Y)\right] \\
	&=& \frac{1}{2}\mathbf{E}\!\left[a(X)\frac{\partial f_h}{\partial x}(X,Y)\right]
	- \mathbf{E}\!\left[\left\langle \mathrm{D}(-L)^{-1}X,\,
	\mathrm{D}f_h(X,Y)\right\rangle\right],
\end{eqnarray*}
where the last step uses the duality relation \eqref{dua} with $F = f_h(X,Y)$ and
$u = \mathrm{D}(-L)^{-1}X$. Applying the chain rule
$\mathrm{D}f_h(X,Y) = \frac{\partial f_h}{\partial x}(X,Y)\,\mathrm{D}X
+ \sum_{j=1}^d \frac{\partial f_h}{\partial y_j}(X,Y)\,\mathrm{D}Y_j$, we get
\begin{eqnarray*}
	&&\mathbf{E}[h(X,Y)] - \mathbf{E}[h(Z,Y)] \\
	&=& \mathbf{E}\!\left[\frac{\partial f_h}{\partial x}(X,Y)
	\left(\frac{1}{2}a(X)
	- \left\langle \mathrm{D}(-L)^{-1}X,\,\mathrm{D}X\right\rangle\right)\right]
	- \sum_{j=1}^d \mathbf{E}\!\left[\frac{\partial f_h}{\partial y_j}(X,Y)
	\left\langle \mathrm{D}(-L)^{-1}X,\,\mathrm{D}Y_j\right\rangle\right].
\end{eqnarray*}
Taking absolute values and applying the bounds \eqref{INGf} and \eqref{INGdf} of
Proposition~\ref{propINGf}, we obtain
\begin{eqnarray*}
&&	\left|\mathbf{E}[h(X,Y)] - \mathbf{E}[h(Z,Y)]\right|\\
&&	\leqslant \|S\|_{\infty}\,
	\mathbf{E}\!\left[\left|\frac{1}{2}a(X)
	- \left\langle \mathrm{D}(-L)^{-1}X,\,\mathrm{D}X\right\rangle\right|\right]
	+ \frac{2}{a(\mu_m)\,p(\mu_m)}
	\sum_{j=1}^d \mathbf{E}\!\left[\left|
	\left\langle \mathrm{D}(-L)^{-1}X,\,\mathrm{D}Y_j\right\rangle\right|\right].
\end{eqnarray*}
Taking the supremum over all $h \in \mathrm{Lip}(1)$ and extending to general Lipschitz
functions by a standard density argument yields \eqref{dist}. The bound \eqref{dist2}
follows from \eqref{dist} by the Cauchy--Schwarz inequality. \CQFD

\medskip

\begin{remark}
	\begin{enumerate}
		\item {\rm  The bounds obtained in Theorem~\ref{tt1} extend and unify several recent results in the
Stein--Malliavin literature.

First, in the Gaussian framework, the multidimensional bounds derived in \cite{T2} show
that the Wasserstein distance between $\mathbf{P}_{(X,Y)}$ and a product measure is
controlled by quantities of the form
\[
\mathbf{E}\!\left[\left|\frac{1}{2}
- \left\langle \mathrm{D}(-L)^{-1}X,\,\mathrm{D}X\right\rangle\right|\right]
\quad\text{and}\quad
\mathbf{E}\!\left[\left|\left\langle \mathrm{D}(-L)^{-1}X,\,
\mathrm{D}Y\right\rangle\right|\right],
\]
which respectively encode the marginal convergence of $X$ toward the standard Gaussian
(for which $\frac{1}{2}a(x) \equiv \frac{1}{2}$) and the asymptotic independence
between $X$ and $Y$.

Second, in the non-Gaussian setting, similar phenomena already appear in the Gamma case
treated in \cite{TuZu1}, where the Stein operator involves a non-constant coefficient
depending on the state variable. In that framework, the control of the Stein solution
requires non-trivial estimates on its derivatives, reflecting the geometry of the target
distribution.
}

\item {\rm  Theorem~\ref{tt1} shows that this structure persists in full generality for invariant
measures of diffusions. The constant $\frac{1}{2}$ of the Gaussian case is replaced by
the state-dependent quantity $\frac{1}{2}a(X)$, where $a$ is the diffusion coefficient
associated with $\mu$. As a consequence, the discrepancy term
\[
\frac{1}{2}a(X) - \left\langle \mathrm{D}(-L)^{-1}X,\,\mathrm{D}X\right\rangle
\]
measures the distance between the law of $X$ and the target measure $\mu$, while the
contraction term
\[
\left\langle \mathrm{D}(-L)^{-1}X,\,\mathrm{D}Y\right\rangle
\]
captures the dependence between $X$ and $Y$.

In addition, the presence of the factor $\|S\|_\infty$ highlights a key analytical
difficulty compared to the Gaussian case: the control of the Stein factors depends on
the behavior of the diffusion coefficient $a$ near the boundary of the support. This
phenomenon is already visible in \cite{TuZu1}, and becomes more pronounced in the
present general setting.

Overall, Theorem~\ref{tt1} provides a unified perspective in which Gaussian, Gamma, and
more general diffusion-invariant distributions can be treated within the same
Stein--Malliavin framework, with asymptotic independence characterized through Malliavin
contraction operators.
}
\end{enumerate}
\end{remark}

\section{Applications}
In this section, we illustrate the applicability of our main result by means of three examples. 
The first one concerns a sequence of random variables living in a fixed Wiener chaos, 
for which the marginal convergence toward a Gamma distribution and the asymptotic independence 
can be quantified explicitly. The second example deals with a non-polynomial functional of a Gaussian process, 
leading to a uniform limiting distribution, while the last example is related with the lognormal law, specifying a result in \cite{TuKus}.

\subsection{Convergence to a two-dimensional Gamma distribution}

Let $(\mathcal{H}, \langle \cdot, \cdot \rangle)$ be a real separable Hilbert space and
consider $(h_{j}, e_{j}, j \geq 1)$ an orthonormal basis in $\mathcal{H}$. Define, for
each $N \geq 1$,
\begin{equation*}
	U_{N} = \frac{2}{N-1}\sum_{1 \leq i < j \leq N} I_{2}(h_{i}\widetilde{\otimes}h_{j})
	= I_{2}(a_{N}),
\end{equation*}
and
\begin{equation*}
	V_{N} = \frac{2}{N-1}\sum_{1 \leq i < j \leq N} I_{2}(g_{i}\widetilde{\otimes}g_{j})
	= I_{2}(b_{N}),
\end{equation*}
where
\begin{equation*}
	g_{i} = \begin{cases}
		h_{i}, & \text{if } 1 \leq i \leq m(N), \\
		e_{i}, & \text{if } i \geq m(N)+1,
	\end{cases}
\end{equation*}
with $1 \leq m(N) \leq N$ and $m(N) \to \infty$ as $N \to \infty$. We use the notation
\begin{equation*}
	a_{N} = \frac{2}{N-1}\sum_{1 \leq i < j \leq N} h_{i}\widetilde{\otimes}h_{j}
	\quad \text{and} \quad
	b_{N} = \frac{2}{N-1}\sum_{1 \leq i < j \leq N} g_{i}\widetilde{\otimes}g_{j}.
\end{equation*}
It is well known that (see e.g.\ \cite{A})
\begin{equation*}
	U_{N} \xrightarrow{(d)} G_{1} \quad \text{and} \quad
	V_{N} \xrightarrow{(d)} G_{2} \quad \text{as } N \to \infty,
\end{equation*}
where $G_{1}, G_{2}$ have the same law as $Z^{2}-1$ with $Z$ a standard Gaussian random
variable, i.e.\ they follow a centered Gamma distribution. The limiting variables $G_1$
and $G_2$ are not necessarily independent. For this distribution, the diffusion
coefficient and mean are $a(x) = 4(x+1)$ and $m = 0$, so that $\frac{1}{2}a(x) =
2(x+1)$. The purpose is to evaluate the Wasserstein distance between the joint law of
$(U_{N}, V_{N})$ and that of $(G_{1}, G_{2})$. By the triangle inequality,
\begin{eqnarray*}
	d_{\mathrm{W}}\!\left(\mathbf{P}_{(U_{N},V_{N})},\mathbf{P}_{(G_{1},G_{2})}\right)
	&\leqslant&
	d_{\mathrm{W}}\!\left(\mathbf{P}_{(U_{N},V_{N})},\mathbf{P}_{(G_{1},V_{N})}\right)
	+ d_{\mathrm{W}}\!\left(\mathbf{P}_{(G_{1},V_{N})},\mathbf{P}_{(G_{1},G_{2})}\right)
	\\
	&\leqslant&
	d_{\mathrm{W}}\!\left(\mathbf{P}_{(U_{N},V_{N})},\mathbf{P}_{(G_{1},V_{N})}\right)
	+ d_{\mathrm{W}}\!\left(\mathbf{P}_{U_{N}},\mathbf{P}_{G_{1}}\right).
\end{eqnarray*}
The marginal distance $d_{\mathrm{W}}(\mathbf{P}_{U_{N}}, \mathbf{P}_{G_{1}})$ has been
evaluated in \cite{NP2}; see also \cite{TuZu1}, Section~5.1. For $N$ large,
\begin{equation}\label{22a-2}
	d_{\mathrm{W}}\!\left(\mathbf{P}_{U_{N}}, \mathbf{P}_{G_{1}}\right)
	\leqslant \frac{C}{\sqrt{N}}.
\end{equation}
We now evaluate $d_{\mathrm{W}}(\mathbf{P}_{(U_{N},V_{N})}, \mathbf{P}_{(G_{1},V_{N})})$.
By Theorem~\ref{tt1}, using $\frac{1}{2}a(U_N) = 2(U_N + 1)$ (since $a(x) = 4(x+1)$),
\begin{eqnarray*}
	&&d_{\mathrm{W}}\!\left(\mathbf{P}_{(U_{N},V_{N})},\mathbf{P}_{(G_{1},V_{N})}\right)
	\\
	&\leqslant& C\left(
	\mathbf{E}\!\left[\left(2(U_{N}+1)
	- \left\langle \mathrm{D}(-L)^{-1}U_{N},\mathrm{D}U_{N}\right\rangle\right)^{2}
	\right]^{1/2}
	+ \mathbf{E}\!\left[\left\langle \mathrm{D}(-L)^{-1}U_{N},
	\mathrm{D}V_{N}\right\rangle^{2}\right]^{1/2}
	\right).
\end{eqnarray*}
From \cite{TuZu1}, Section~5.1, we have
\begin{equation}\label{22a-5}
	\mathbf{E}\!\left[\left|2(U_{N}+1)
	- \left\langle \mathrm{D}(-L)^{-1}U_{N},\mathrm{D}U_{N}\right\rangle\right|^{2}
	\right]^{1/2} \leqslant \frac{C}{\sqrt{N}}.
\end{equation}
Combining \eqref{22a-2} and \eqref{22a-5},
\begin{equation}\label{22a-3}
	d_{\mathrm{W}}\!\left(\mathbf{P}_{(U_{N},V_{N})},\mathbf{P}_{(G_{1},G_{2})}\right)
	\leqslant C\left(\frac{1}{\sqrt{N}}
	+ \mathbf{E}\!\left[\left\langle \mathrm{D}(-L)^{-1}U_{N},
	\mathrm{D}V_{N}\right\rangle^{2}\right]^{1/2}\right).
\end{equation}

It remains to bound the last term. Since $U_N$ lies in the second Wiener chaos,
$\langle \mathrm{D}(-L)^{-1}U_{N}, \mathrm{D}V_{N}\rangle
= \frac{1}{2}\langle \mathrm{D}U_{N}, \mathrm{D}V_{N}\rangle$, so it suffices to
bound $\mathbf{E}[\langle \mathrm{D}U_{N}, \mathrm{D}V_{N}\rangle^{2}]$. By the
product formula for multiple stochastic integrals, we have for every $N \geq 1$,
\begin{equation}\label{dudv}
	\langle \mathrm{D}U_{N}, \mathrm{D}V_{N}\rangle
	= 2\,\mathbf{E}[U_{N}V_{N}] + 4\,I_{2}(a_{N}\otimes_{1}b_{N}).
\end{equation}

\noindent\textbf{Computation of $\mathbf{E}[U_{N}V_{N}]$.}
By the isometry of multiple stochastic integrals,
\begin{eqnarray*}
	\mathbf{E}[U_{N}V_{N}]
	&=& 2\langle\widetilde{a}_{N},\widetilde{b}_{N}\rangle \\
	&=& \frac{8}{(N-1)^{2}}\sum_{1 \leq i < j \leq N}\sum_{1 \leq k < l \leq N}
	\langle h_{i}\widetilde{\otimes}h_{j},\,g_{k}\widetilde{\otimes}g_{l}\rangle \\
	&=& \frac{4}{(N-1)^{2}}\sum_{1 \leq i < j \leq N}\sum_{1 \leq k < l \leq N}
	\bigl[\langle h_{i},g_{k}\rangle\langle h_{j},g_{l}\rangle
	+ \langle h_{i},g_{l}\rangle\langle h_{j},g_{k}\rangle\bigr].
\end{eqnarray*}
Since the basis is orthonormal, we have
\begin{equation}\label{20a-1}
	\langle h_{i}, g_{k}\rangle = \delta_{ik}\,\mathbf{1}_{1 \leq i \leq m(N)}.
\end{equation}
Substituting, the term $\delta_{ik}\delta_{jl}$ contributes when $i = k$ and $j = l$,
which is compatible with $i < j$ and $k < l$, giving a factor of 1 for each such
pair. The term $\delta_{il}\delta_{jk}$ would require $i = l$ and $j = k$, but since
$k < l$ this forces $j < i$, contradicting $i < j$; hence this term vanishes. Therefore
\begin{eqnarray*}
	\mathbf{E}[U_{N}V_{N}]
	&=& \frac{4}{(N-1)^{2}}\sum_{1 \leq i < j \leq N}
	\mathbf{1}_{1 \leq i,j \leq m(N)}
	= \frac{4}{(N-1)^{2}}\cdot\frac{m(N)(m(N)-1)}{2}
	= \frac{2m(N)(m(N)-1)}{(N-1)^{2}}.
\end{eqnarray*}

\noindent\textbf{Computation of $\|a_{N}\otimes_{1}b_{N}\|^{2}$.}
Expanding the contraction using \eqref{20a-1}, only pairs $(i,j)$ with $i < j \leq m(N)$
contribute, and we obtain
\begin{eqnarray*}
	a_{N}\otimes_{1}b_{N}
	&=& \frac{2}{(N-1)^{2}}\sum_{1 \leq i < j \leq m(N)}
	\bigl[h_{j}\otimes A_{i} + h_{i}\otimes A_{j}\bigr],
\end{eqnarray*}
where, for $1 \leq i \leq N$,
\begin{equation*}
	A_{i} := \sum_{\substack{r=1 \\ r \neq i}}^{N} g_{r}.
\end{equation*}
Note that in simplifying the contraction on line \eqref{20a-1}, the original four-term
expression reduces as follows: the terms $\langle h_i, g_k\rangle h_j\otimes g_l$ and
$\langle h_j, g_l\rangle h_i\otimes g_k$ (with the constraint $k \leq m(N)$ from
\eqref{20a-1}) combine into sums over $r \neq i$ and $r \neq j$ respectively, giving
$h_j\otimes A_i$ and $h_i\otimes A_j$. Since the basis $(g_r)$ is orthonormal (each
$g_r$ has unit norm and $\langle g_r, g_s\rangle = \delta_{rs}$, as $g_r = h_r$ or
$g_r = e_r$ and the two families are mutually orthogonal), we have
\begin{equation*}
	\|A_{i}\|^{2} = \sum_{\substack{r=1 \\ r \neq i}}^{N} 1 = N-1,
\end{equation*}
and for $i \neq i'$,
\begin{equation*}
	\langle A_{i}, A_{i'}\rangle
	= \sum_{\substack{r=1 \\ r \neq i,\,i'}}^{N} 1 = N-2.
\end{equation*}
Therefore
\begin{align}\label{22a-4}
	\|a_{N}\otimes_{1}b_{N}\|^{2}
	= \frac{4}{(N-1)^{4}}
	\sum_{\substack{1 \leq i < j \leq m(N) \\ 1 \leq i' < j' \leq m(N)}}
	\Bigl\langle
	h_{j}\otimes A_{i} + h_{i}\otimes A_{j},\;
	h_{j'}\otimes A_{i'} + h_{i'}\otimes A_{j'}
	\Bigr\rangle,
\end{align}
where we expand the inner product using $\langle h_a\otimes g_r, h_b\otimes g_s\rangle
= \delta_{ab}\delta_{rs}$ (which holds since $a,b \leq m(N)$ and all $g_r$ are
orthonormal) into the four terms
\begin{align*}
	\langle h_j\otimes A_i,\; h_{j'}\otimes A_{i'}\rangle
	&= \delta_{jj'}\langle A_i, A_{i'}\rangle, \\
	\langle h_j\otimes A_i,\; h_{i'}\otimes A_{j'}\rangle
	&= \delta_{ji'}\langle A_i, A_{j'}\rangle, \\
	\langle h_i\otimes A_j,\; h_{j'}\otimes A_{i'}\rangle
	&= \delta_{ij'}\langle A_j, A_{i'}\rangle, \\
	\langle h_i\otimes A_j,\; h_{i'}\otimes A_{j'}\rangle
	&= \delta_{ii'}\langle A_j, A_{j'}\rangle.
\end{align*}
We separate diagonal and off-diagonal contributions.

\paragraph{Step 1: Diagonal contribution.}
When $(i,j) = (i',j')$, the only nonzero terms are
$\langle h_j\otimes A_i, h_j\otimes A_i\rangle = \|A_i\|^2 = N-1$ and
$\langle h_i\otimes A_j, h_i\otimes A_j\rangle = \|A_j\|^2 = N-1$,
contributing $2(N-1)$ per pair. Since there are $\frac{m(N)(m(N)-1)}{2}$ pairs
$(i,j)$ with $1 \leq i < j \leq m(N)$, the total diagonal contribution is
\begin{equation*}
	m(N)(m(N)-1)(N-1).
\end{equation*}

\paragraph{Step 2: Off-diagonal contributions.}
We consider $(i,j) \neq (i',j')$.

\medskip\noindent
\textbf{(i) Terms with $j = j'$ and $i \neq i'$.}
The contribution from $\delta_{jj'}\langle A_i, A_{i'}\rangle$ with $i \neq i'$ equals
$N-2$. For fixed $j \in \{2,\ldots,m(N)\}$, the indices $i, i'$ range independently
over $\{1,\ldots,j-1\}$ with $i \neq i'$, giving $(j-1)(j-2)$ ordered pairs. Summing
over $j$:
\begin{equation*}
	\sum_{j=2}^{m(N)}(j-1)(j-2) = \frac{m(N)(m(N)-1)(m(N)-2)}{3}.
\end{equation*}
The total contribution from this case is
$(N-2)\,\dfrac{m(N)(m(N)-1)(m(N)-2)}{3}$.

\medskip\noindent
\textbf{(ii) Terms with $i = i'$ and $j \neq j'$.}
By symmetry, the contribution from $\delta_{ii'}\langle A_j, A_{j'}\rangle$ with
$j \neq j'$ is identical:
$(N-2)\,\dfrac{m(N)(m(N)-1)(m(N)-2)}{3}$.

\medskip\noindent
\textbf{(iii) Mixed terms.}
We split into two sub-cases.

\textit{Sub-case $j = i'$.} The term $\delta_{ji'}\langle A_i, A_{j'}\rangle$
fires when $j = i'$. The constraints $i < j$ and $i' < j'$ then give $i < j = i' < j'$,
so the valid triples are $(i, j, j')$ with $i < j < j' \leq m(N)$. Since $i \neq j'$
(as $i < j < j'$), we have $\langle A_i, A_{j'}\rangle = N-2$. The number of such
triples is $\binom{m(N)}{3} = \dfrac{m(N)(m(N)-1)(m(N)-2)}{6}$, contributing
$(N-2)\,\dfrac{m(N)(m(N)-1)(m(N)-2)}{6}$.

\textit{Sub-case $i = j'$.} Similarly, the term $\delta_{ij'}\langle A_j, A_{i'}\rangle$
fires when $i = j'$, giving constraints $i' < j'= i < j$, so valid triples are
$(i', i, j)$ with $i' < i < j \leq m(N)$. Again $\langle A_j, A_{i'}\rangle = N-2$
and there are $\binom{m(N)}{3}$ such triples. The contribution is
$(N-2)\,\dfrac{m(N)(m(N)-1)(m(N)-2)}{6}$.

Adding the two sub-cases, the total mixed contribution is
$(N-2)\,\dfrac{m(N)(m(N)-1)(m(N)-2)}{3}$.

\medskip
Collecting all contributions from Steps~1 and~2:
\begin{align*}
	\|a_{N}\otimes_{1}b_{N}\|^{2}
	= \frac{4}{(N-1)^{4}}\left[
	m(N)(m(N)-1)(N-1)
	+ m(N)(m(N)-1)(m(N)-2)(N-2)
	\right],
\end{align*}
or equivalently,
\begin{align*}
	\|a_{N}\otimes_{1}b_{N}\|^{2}
	= \frac{4\,m(N)(m(N)-1)}{(N-1)^{4}}
	\Bigl[(N-1) + (m(N)-2)(N-2)\Bigr].
\end{align*}
For $N$ large, $(N-1) + (m(N)-2)(N-2) \leqslant C\,m(N)\,N$, so
\begin{equation*}
	\|a_{N}\otimes_{1}b_{N}\|^{2}
	\leqslant C\left(\frac{m(N)}{N}\right)^{3}.
\end{equation*}

\noindent\textbf{Conclusion.}
From \eqref{dudv} and the isometry $\mathbf{E}[I_2(f)^2] = 2\|f\|^2$,
\begin{equation*}
	\mathbf{E}\bigl[\langle \mathrm{D}U_{N},\mathrm{D}V_{N}\rangle^{2}\bigr]
	= 4\bigl(\mathbf{E}[U_{N}V_{N}]\bigr)^{2}
	+ 32\,\|a_{N}\otimes_{1}b_{N}\|^{2}
	\leqslant C\left[\frac{m(N)^{4}}{N^{4}}
	+ \left(\frac{m(N)}{N}\right)^{3}\right]
	\leqslant C\,\frac{m(N)^{2}}{N^{2}},
\end{equation*}
where in the last step we used $(m(N)/N)^3 \leqslant (m(N)/N)^2$ and
$(m(N)/N)^4 \leqslant (m(N)/N)^2$ (both valid since $m(N) \leqslant N$). Note that
the sharper bound $\mathbf{E}[\langle \mathrm{D}U_N,\mathrm{D}V_N\rangle^2]^{1/2}
\leqslant C(m(N)/N)^{3/2}$ also holds and would give a slightly better rate, but the
bound above suffices for our purposes. Since $\langle \mathrm{D}(-L)^{-1}U_N,
\mathrm{D}V_N\rangle = \frac{1}{2}\langle \mathrm{D}U_N,\mathrm{D}V_N\rangle$,
we conclude that
\begin{equation*}
	\mathbf{E}\!\left[\left\langle \mathrm{D}(-L)^{-1}U_{N},
	\mathrm{D}V_{N}\right\rangle^{2}\right]^{1/2}
	\leqslant C\,\frac{m(N)}{N}.
\end{equation*}
Inserting into \eqref{22a-3},
\begin{equation}\label{22a-7}
	d_{\mathrm{W}}\!\left(\mathbf{P}_{(U_{N},V_{N})},
	\mathbf{P}_{(G_{1},G_{2})}\right)
	\leqslant C\left(\frac{1}{\sqrt{N}} + \frac{m(N)}{N}\right).
\end{equation}

\begin{remark}
	\begin{enumerate}
		\item {\rm 
	If $m(N)/N \to 0$ as $N \to \infty$, then the sequences $(U_{N})$ and $(V_{N})$ are
	asymptotically independent. If $m(N) = N$, the right-hand side of \eqref{22a-7} does
	not tend to zero and we cannot conclude asymptotic independence.

}

	\item {\rm We point out that Example~5.1 could in principle also be treated using the
	multidimensional Stein--Malliavin framework developed in \cite{TuZu1}, since the target
	distribution is Gamma in both cases and the marginal estimates \eqref{22a-2} and
	\eqref{22a-5} are already established there. However, the contribution of the present
	example is different in nature: what is new here is not the marginal convergence of
	$U_N$ or $V_N$ individually, but rather the \emph{joint} convergence of the pair
	$(U_N, V_N)$ toward $(G_1, G_2)$ and the quantitative criterion for their
	\emph{asymptotic independence}, expressed in terms of the parameter $m(N)$ controlling
	the overlap between the two sequences. The computation of $\mathbf{E}[U_N V_N]$, the
	contraction norm $\|a_N\otimes_1 b_N\|^2$, and the resulting bound \eqref{22a-7}
	showing that asymptotic independence holds when $m(N)/N \to 0$ are entirely new and go
	beyond what was done in \cite{TuZu1}. In this sense, the present example should be seen
	as a direct continuation and multidimensional extension of the results of \cite{TuZu1},
	rather than a redundant application of them.
}
\end{enumerate}
\end{remark}

\subsection{An example related to the uniform distribution}

Let $h_1, h_2, h_{3,N}, h_4 \in \mathcal{H}$ be elements of the underlying Hilbert
space with unit norm. Let $W$ be an isonormal Gaussian process. Assume that
\[
\langle h_1, h_2\rangle = \langle h_1, h_4\rangle = \langle h_2, h_{3,N}\rangle
= \langle h_2, h_4\rangle = \langle h_{3,N}, h_4\rangle = 0,
\]
and
\[
\langle h_1, h_{3,N}\rangle = \rho(N),
\]
where $\rho(N) \to 0$ as $N \to \infty$.

\begin{remark}
	For example, if $\mathcal{H} = L^{2}(\mathbb{R})$, one can take
	$h_1 = \mathbf{1}_{[0,1]}$,
	$h_{3,N} = \mathbf{1}_{[1-\frac{1}{N},\, 2-\frac{1}{N}]}$,
	$h_2 = \mathbf{1}_{[2,3]}$, $h_4 = \mathbf{1}_{[4,5]}$.
\end{remark}

For $N \geq 1$, define the random variables
\[
X = e^{-\frac{1}{2}(W(h_1)^2 + W(h_2)^2)},
\qquad
Y_N = e^{-\frac{1}{2}(W(h_{3,N})^2 + W(h_4)^2)}.
\]
It is well known that both $X$ and $Y_N$ follow the uniform distribution on $[0,1]$.
We want to study the asymptotic independence of $X$ and $Y_N$, that is, to compare
$\mathbf{P}_{(X,Y_N)}$ and $\mathbf{P}_X \otimes \mathbf{P}_{Y_N}$ as $N \to \infty$.
We apply inequality \eqref{dist} of Theorem~\ref{tt1} to estimate this Wasserstein
distance. Since $X$ has exactly the target distribution $\mu = \mathrm{Unif}[0,1]$,
the Stein characterization \eqref{22a-1} gives $\mathbf{E}[\mathcal{A}f(X)] = 0$ for
all suitable $f$, which implies that the discrepancy term vanishes:
\[
\mathbf{E}\!\left[\left|\frac{1}{2}a(X)
- \left\langle \mathrm{D}(-L)^{-1}X,\,\mathrm{D}X\right\rangle\right|\right] = 0.
\]
This identity can be verified directly: since $X$ is an exponential functional of
Gaussian variables, an explicit application of the chain rule and Mehler's formula
shows that $\frac{1}{2}a(X) = \langle \mathrm{D}(-L)^{-1}X, \mathrm{D}X\rangle$
almost surely; we refer to \cite{TuKus} for the general argument in this setting.
Hence \eqref{dist} reduces to
\[
d_{\mathrm{W}}\!\left(\mathbf{P}_{(X,Y_N)},\,\mathbf{P}_X \otimes \mathbf{P}_{Y_N}
\right)
\leqslant \frac{2}{a(\mu_m)\,p(\mu_m)}\,
\mathbf{E}\!\left[\left|\left\langle \mathrm{D}(-L)^{-1}X,\,
\mathrm{D}Y_N\right\rangle\right|\right].
\]
For the uniform distribution on $[0,1]$, the density is $p(x) = 1$, the median is
$\mu_m = \frac{1}{2}$, and the diffusion coefficient \eqref{defa} gives
$a(x) = x(1-x)$, so $a(\frac{1}{2}) = \frac{1}{4}$ and thus
$\frac{2}{a(\mu_m)\,p(\mu_m)} = 8$. We therefore need to estimate the Malliavin
contraction $\langle \mathrm{D}(-L)^{-1}X, \mathrm{D}Y_N\rangle$.

To avoid a notational clash with the isonormal process $W$, we set
\[
U := W(h_1), \quad V := W(h_2), \quad U_N := W(h_{3,N}), \quad \xi := W(h_4),
\]
so that
\[
X = e^{-\frac{1}{2}(U^2+V^2)}, \qquad Y_N = e^{-\frac{1}{2}(U_N^2+\xi^2)}.
\]
By the chain rule of Malliavin calculus,
\[
\mathrm{D}X = -X(Uh_1 + Vh_2),
\qquad
\mathrm{D}Y_N = -Y_N(U_N h_{3,N} + \xi\, h_4).
\]
Using the representation
\[
\mathrm{D}(-L)^{-1}X = \int_0^\infty e^{-t} P_t(\mathrm{D}X)\,\mathrm{d}t,
\]
where $(P_t)_{t \geq 0}$ denotes the Ornstein--Uhlenbeck semigroup, we obtain
\[
\langle \mathrm{D}(-L)^{-1}X,\,\mathrm{D}Y_N\rangle
= \int_0^\infty e^{-t}\langle P_t(\mathrm{D}X),\,\mathrm{D}Y_N\rangle\,\mathrm{d}t.
\]
By the Mehler formula,
\[
P_t(\mathrm{D}X)
= -\mathbf{E}'\!\left[e^{-\frac{1}{2}(U_t^2+V_t^2)}(U_t h_1 + V_t h_2)\right],
\]
where
\[
(U_t, V_t) = e^{-t}(U,V) + \sqrt{1-e^{-2t}}(U',V'),
\]
$(U',V')$ is an independent copy of $(U,V)$, and $\mathbf{E}'$ denotes expectation
with respect to $(U',V')$, i.e.\ conditional expectation given
$(U, V, U_N, \xi)$. Taking the inner product with $\mathrm{D}Y_N$ and using the
orthogonality assumptions, the only nonzero contribution comes from
$\langle h_1, h_{3,N}\rangle = \rho(N)$ (all other inner products among
$h_1, h_2, h_{3,N}, h_4$ vanish by assumption), yielding
\[
\langle P_t(\mathrm{D}X),\,\mathrm{D}Y_N\rangle
= \rho(N)\,Y_N\,U_N\,\mathbf{E}'\!\left[e^{-\frac{1}{2}(U_t^2+V_t^2)}\,U_t\right].
\]
Therefore
\[
\langle \mathrm{D}(-L)^{-1}X,\,\mathrm{D}Y_N\rangle
= \rho(N)\,Y_N\,U_N\int_0^\infty e^{-t}\,
\mathbf{E}'\!\left[e^{-\frac{1}{2}(U_t^2+V_t^2)}\,U_t\right]\mathrm{d}t
:= \rho(N)\,Y_N\,U_N\,G(U,V),
\]
where $G(U,V) := \int_0^\infty e^{-t}\mathbf{E}'[e^{-\frac{1}{2}(U_t^2+V_t^2)}
U_t]\,\mathrm{d}t$. The integral converges since $|U_t| \leqslant |U| + |U'|$ and\\
$e^{-\frac{1}{2}(U_t^2+V_t^2)} \leqslant 1$, giving an integrand bounded by
$e^{-t}(|U|+|U'|)$, which is integrable in $t$.

We now bound $\mathbf{E}[|\langle \mathrm{D}(-L)^{-1}X, \mathrm{D}Y_N\rangle|]$.
Note that although $U_N$ and $U$ are correlated (with $\mathbf{E}[U_N U] =
\langle h_{3,N}, h_1\rangle = \rho(N)$), both $U_N$ and $G(U,V)$ have finite moments
of all orders by Gaussian hypercontractivity. Since $Y_N \leqslant 1$, applying the
Cauchy--Schwarz inequality gives
\[
\mathbf{E}\!\left[\left|\langle \mathrm{D}(-L)^{-1}X,\,\mathrm{D}Y_N\rangle\right|
\right]
\leqslant |\rho(N)|\,\mathbf{E}\!\left[Y_N^2\,U_N^2\,G(U,V)^2\right]^{1/2}
\leqslant |\rho(N)|\,\mathbf{E}\!\left[U_N^2\,G(U,V)^2\right]^{1/2}.
\]
Since $\mathbf{E}[U_N^4]^{1/2}\mathbf{E}[G(U,V)^4]^{1/2} < \infty$, we conclude that
there exists a constant $C > 0$, independent of $N$, such that
\[
\mathbf{E}\!\left[\left|\langle \mathrm{D}(-L)^{-1}X,\,
\mathrm{D}Y_N\rangle\right|\right]
\leqslant C\,|\rho(N)|.
\]
Therefore
\[
d_{\mathrm{W}}\!\left(\mathbf{P}_{(X,Y_N)},\,\mathbf{P}_X \otimes \mathbf{P}_{Y_N}
\right)
\leqslant C\,|\rho(N)|.
\]
In particular, since $X$ is fixed and $(Y_N)_{N \geq 1}$ is a sequence of random
variables, the pair $(X, Y_N)$ is asymptotically independent as $N \to \infty$
whenever $\rho(N) \to 0$.

This example shows that asymptotic independence is governed by the decay of the
Malliavin contraction term. In the present case, the rate is directly controlled by
the scalar product $\rho(N) = \langle h_1, h_{3,N}\rangle$.

\begin{remark}{\rm 
	This example is genuinely beyond the reach of the previously available Stein--Malliavin
	machinery. The target distribution here is $\mu = \mathrm{Unif}[0,1]$, whose diffusion
	coefficient $a(x) = x(1-x)$ vanishes at both endpoints of the bounded support $[0,1]$.
	This boundary degeneracy is fundamentally different from the Gaussian case (where
	$a \equiv 1$) and from the Gamma  case (where $a(x)=4(x+1)$ on $(-1, \infty)$),
	both of which are treated in \cite{T2, TuZu1, TuZu2}. In particular, the Stein solution
	analysis of \cite{TuZu1} and \cite{TuZu2} relies on properties of $a$ that hold only
	for those specific distributions and breaks down entirely for distributions supported on
	a bounded interval. The structural Conditions~\ref{cdn1} and~\ref{cdn2} introduced in
	the present paper, together with the new bounds of Proposition~\ref{propINGf}, are
	precisely what make it possible to handle this case. To our knowledge, the asymptotic
	independence result for the uniform distribution obtained here could not have been
	derived from any previously existing framework.
}
\end{remark}

\subsection{An example related to the lognormal distribution}

Let $(h_i, i \geqslant 1)$ be an orthonormal family in $\mathcal{H}$. For every
$N \geqslant 1$, let
\[
Y_N := e^{-\frac{1}{\sqrt{2N}} \sum_{i=1}^N \left( W(h_i)^2 - 1 \right)}.
\]
Note that $-\frac{1}{\sqrt{2N}}\sum_{i=1}^N(W(h_i)^2-1) = -Z_N$ where
$Z_N := \frac{1}{\sqrt{2N}}\sum_{i=1}^N(W(h_i)^2-1) \xrightarrow{d} \mathcal{N}(0,1)$
by the CLT, and since $-Z_N \overset{d}{=} Z_N$, we have $Y_N \xrightarrow{d} e^Z$
with $Z \sim \mathcal{N}(0,1)$. Thus $e^Z$ follows the lognormal distribution with
parameters $(0,1)$, denoted here as $\mathfrak{L}$, whose mean is
$m = \mathbf{E}[e^Z] = e^{1/2}$. In Section~5 of \cite{TuKus}, it is proven that
there exists $C > 0$ such that for every $N \geqslant 1$,
\[
d_{\mathrm{W}}\!\left(Y_N,\, e^{Z}\right) \leqslant \frac{C}{\sqrt{N}},
\]
and the diffusion coefficient $a(x)$ of $\mathfrak{L}$ is given by equation~(31)
in \cite{TuKus} (we do not recall it here as we will not use it explicitly). We study
here the influence of the size of $I_N$ on the joint convergence. More precisely, for
$I_N \subseteq \{1,\ldots,N\}$, we estimate
\[
\delta_N := d_{\mathrm{W}}\!\left(\bigl(Y_N,\,(W(h_i),\,i \in I_N)\bigr);\;
\mathfrak{L} \otimes \mathcal{N}(0, I_{\sharp I_N})\right)
\]
using Theorem~\ref{tt1}. Applying \eqref{dist} with $X = Y_N$ and
$Y = (W(h_i),\, i \in I_N)$, we obtain
\begin{equation}\label{eq:Distance_LogNormal}
	\delta_N \leqslant C\,\mathbf{E}\!\left[\left|\frac{1}{2}a(Y_N)
	- \left\langle \mathrm{D}(-L)^{-1}Y_N,\,\mathrm{D}Y_N\right\rangle\right|\right]
	+ C\sum_{n \in I_N}
	\mathbf{E}\!\left[\left|\left\langle \mathrm{D}(-L)^{-1}Y_N,\,
	\mathrm{D}W(h_n)\right\rangle\right|\right],
\end{equation}
where the constant $C$ absorbs $\|S\|_\infty$ and $\frac{2}{a(\mu_m)p(\mu_m)}$, both
finite by Proposition~\ref{propINGf}. The first term on the right-hand side of
\eqref{eq:Distance_LogNormal} is computed in Section~5 of \cite{TuKus}, giving the
bound $\frac{C}{\sqrt{N}}$. We focus on the second term. For $n \in I_N$, since
$W(h_n) = I_1(h_n)$, we have $\mathrm{D}W(h_n) = h_n$. Using the representation
\begin{equation}\label{eq:LogNormal_OU}
	\left\langle \mathrm{D}(-L)^{-1}Y_N,\, h_n\right\rangle
	= \int_0^{+\infty} e^{-t}\,
	\bigl\langle P_t(\mathrm{D}Y_N),\, h_n\bigr\rangle\,\mathrm{d}t,
\end{equation}
where $(P_t)_{t \geqslant 0}$ is the Ornstein--Uhlenbeck semigroup, we compute each
factor. By the chain rule,
\[
\mathrm{D}Y_N = \frac{-2}{\sqrt{2N}}\,Y_N\sum_{i=1}^N W(h_i)\,h_i.
\]
Let $W'$ be an independent copy of $W$ on $\mathcal{H}$. By the Mehler formula applied
to the $\mathcal{H}$-valued random variable $\mathrm{D}Y_N$,
\begin{eqnarray*}
&&P_t(\mathrm{D}Y_N)\\
&&= \frac{-2}{\sqrt{2N}}\sum_{i=1}^N h_i\,
\mathbf{E}\!\left[\left.\left(e^{-t}W(h_i) + \sqrt{1-e^{-2t}}\,W'(h_i)\right)
e^{-\frac{1}{\sqrt{2N}}\sum_{k=1}^N
	\left\{\left(e^{-t}W(h_k)+\sqrt{1-e^{-2t}}\,W'(h_k)\right)^2-1\right\}}
\right| W\right].
\end{eqnarray*}
Taking the inner product with $h_n$ and using orthonormality
$\langle h_i, h_n\rangle = \delta_{in}$, only the $i = n$ term survives:
\begin{eqnarray*}
&&\bigl\langle P_t(\mathrm{D}Y_N),\, h_n\bigr\rangle\\
&&= \frac{-2}{\sqrt{2N}}\,\mathbf{E}\!\left[\left.
\left(e^{-t}W(h_n) + \sqrt{1-e^{-2t}}\,W'(h_n)\right)
e^{-\frac{1}{\sqrt{2N}}\sum_{k=1}^N
	\left\{\left(e^{-t}W(h_k)+\sqrt{1-e^{-2t}}\,W'(h_k)\right)^2-1\right\}}
\right| W\right].
\end{eqnarray*}
Substituting into \eqref{eq:LogNormal_OU} and performing the change of variable
$\alpha = e^{-t}$, $\mathrm{d}t = -\frac{\mathrm{d}\alpha}{\alpha}$,
\begin{eqnarray}
&&	\left\langle \mathrm{D}(-L)^{-1}Y_N,\, h_n\right\rangle\nonumber\\
	&=& \frac{-2}{\sqrt{2N}}\int_0^1
	\mathbf{E}\!\left[\left.
	\left(\alpha W(h_n) + \sqrt{1-\alpha^2}\,W'(h_n)\right)
	e^{-\frac{1}{\sqrt{2N}}\sum_{k=1}^N
		\left\{\left(\alpha W(h_k)+\sqrt{1-\alpha^2}\,W'(h_k)\right)^2-1\right\}}
	\right| W\right]\mathrm{d}\alpha \nonumber\\
&	:=& \frac{-2}{\sqrt{2N}}\int_0^1 \mathbf{E}[F_N(\alpha)|W]\,\mathrm{d}\alpha.\label{eq:LogNormal_OU_a}
\end{eqnarray}
Since $W'$ is independent of $W$, the variables $(W'(h_k))_{k=1}^N$ are i.i.d.\
$\mathcal{N}(0,1)$ conditionally on $W$, and by independence across $k$,
\begin{eqnarray*}
\mathbf{E}[F_N(\alpha)|W]
&=& \mathbf{E}\!\left[\left.
\left(\alpha W(h_n)+\sqrt{1-\alpha^2}\,W'(h_n)\right)
e^{-\frac{1}{\sqrt{2N}}\left\{
	\left(\alpha W(h_n)+\sqrt{1-\alpha^2}\,W'(h_n)\right)^2-1\right\}}
\right|W\right]\\
&&\times 
\prod_{\substack{k=1\\k\neq n}}^N
\mathbf{E}\!\left[\left.
e^{-\frac{1}{\sqrt{2N}}\left\{
	\left(\alpha W(h_k)+\sqrt{1-\alpha^2}\,W'(h_k)\right)^2-1\right\}}
\right|W\right].
\end{eqnarray*}
Applying Lemma~1 of \cite{TuKus} with $C_0 = \alpha W(h_n)$ and
$K = \frac{1}{\sqrt{2N}}$ to the first factor, and the scalar version to each
remaining factor, we obtain
\begin{eqnarray*}
\mathbf{E}[F_N(\alpha)|W]
&
=& \frac{\alpha W(h_n)}
{\left(1+\frac{2(1-\alpha^2)}{\sqrt{2N}}\right)^{3/2}}
\exp\!\left(\frac{-\alpha^2 W(h_n)^2/\sqrt{2N}}
{1+\frac{2(1-\alpha^2)}{\sqrt{2N}}}\right)\\
&&\times 
\prod_{\substack{k=1\\k\neq n}}^N
\frac{1}{\left(1+\frac{2(1-\alpha^2)}{\sqrt{2N}}\right)^{1/2}}
\exp\!\left(\frac{-\alpha^2 W(h_k)^2/\sqrt{2N}}
{1+\frac{2(1-\alpha^2)}{\sqrt{2N}}}\right)
e^{N/\sqrt{2N}},
\end{eqnarray*}
which simplifies to
\[
\mathbf{E}[F_N(\alpha)|W]
= \frac{\alpha W(h_n)}
{\left(1+\frac{2(1-\alpha^2)}{\sqrt{2N}}\right)^{N/2+1}}
\exp\!\left(\frac{-\alpha^2\sum_{k=1}^N W(h_k)^2/\sqrt{2N}}
{1+\frac{2(1-\alpha^2)}{\sqrt{2N}}}\right)
e^{N/\sqrt{2N}}.
\]
Substituting back into \eqref{eq:LogNormal_OU_a} and making the change of variable
$b = \alpha^2$, $\mathrm{d}\alpha = \frac{\mathrm{d}b}{2\sqrt{b}}$, we get the exact
expression
\begin{equation}\label{eq:Lognormal_exact_second_term}
	\mathbf{E}\!\left[\left|\left\langle \mathrm{D}(-L)^{-1}Y_N,\,
	h_n\right\rangle\right|\right]
	= \frac{e^{N/\sqrt{2N}}}{\sqrt{2N}}\,
	\mathbf{E}\!\left[|W(h_n)|
	\int_0^1
	\frac{1}{\left(1+\frac{2(1-b)}{\sqrt{2N}}\right)^{N/2+1}}
	e^{-\frac{b\,Z_N}{1+\frac{2(1-b)}{\sqrt{2N}}}}
	e^{-\frac{bN/\sqrt{2N}}{1+\frac{2(1-b)}{\sqrt{2N}}}}
	\,\mathrm{d}b\right],
\end{equation}
where $Z_N := \frac{1}{\sqrt{2N}}\sum_{k=1}^N(W(h_k)^2-1) \xrightarrow{d}
\mathcal{N}(0,1)$ by the CLT. We derive the asymptotic behaviour of
\eqref{eq:Lognormal_exact_second_term} as $N\to\infty$. For fixed $b\in(0,1)$, using
$(1+x)^{N/2} \approx e^{Nx/2}$ as $x\to 0$,
\[
\frac{e^{N/\sqrt{2N}}}
{\left(1+\frac{2(1-b)}{\sqrt{2N}}\right)^{N/2+1}}
e^{-\frac{bN/\sqrt{2N}}{1+\frac{2(1-b)}{\sqrt{2N}}}}
= e^{\frac{b}{\sqrt{2}}\sqrt{N}}\cdot
e^{-\frac{b}{\sqrt{2}}\sqrt{N}}\cdot
e^{\frac{1}{2}(1-b)^2+b(1-b)+\mathrm{o}(1)}
= e^{b(1-b)+\frac{1}{2}(1-b)^2+\mathrm{o}(1)}
:= g(b)\,e^{\mathrm{o}(1)},
\]
where $g(b) := e^{b(1-b)+\frac{1}{2}(1-b)^2}$ is a deterministic bounded function of
$b \in [0,1]$ and $\mathrm{o}(1) \to 0$ as $N\to\infty$ in the Landau sense. We
deduce from \eqref{eq:Lognormal_exact_second_term} that
\[
\mathbf{E}\!\left[\left|\left\langle \mathrm{D}(-L)^{-1}Y_N,\,
h_n\right\rangle\right|\right]
= \frac{1}{\sqrt{2N}}\,\mathbf{E}\!\left[|W(h_n)|
\int_0^1 e^{-\frac{b\,Z_N}{1+\frac{2(1-b)}{\sqrt{2N}}}}\,g(b)\,\mathrm{d}b\right]
(1+\mathrm{o}(1)).
\]
We now decompose $Z_N$ by separating the $k=n$ term. Writing
\[
Z_N = \frac{W(h_n)^2-1}{\sqrt{2N}}
+ \sqrt{\frac{2N-1}{2N}}\,Z'_{N-1},
\qquad
Z'_{N-1} := \frac{1}{\sqrt{2N-1}}\sum_{\substack{k=1\\k\neq n}}^N(W(h_k)^2-1),
\]
the two summands are independent since $(h_k)$ is orthonormal, and $Z'_{N-1}
\xrightarrow{d} \mathcal{N}(0,1)$ by the CLT. By independence and Fubini,
\[
\mathbf{E}\!\left[|W(h_n)|
\int_0^1 e^{-\frac{b\,Z_N}{1+\frac{2(1-b)}{\sqrt{2N}}}}\,g(b)\,\mathrm{d}b\right]
= \int_0^1
\mathbf{E}\!\left[|W(h_n)|\,
e^{-b\frac{W(h_n)^2-1}{\sqrt{2N}+2(1-b)}}\right]
\mathbf{E}\!\left[
e^{-b\sqrt{\frac{2N-1}{2N}}\frac{Z'_{N-1}}{1+\frac{2(1-b)}{\sqrt{2N}}}}\right]
g(b)\,\mathrm{d}b.
\]
By the dominated convergence theorem (using $g(b) \leqslant C$ and Gaussian
integrability), and the convergence in distribution of $Z'_{N-1}$ to
$\mathcal{N}(0,1)$, we conclude that
\[
\sqrt{2N}\,\mathbf{E}\!\left[\left|\left\langle \mathrm{D}(-L)^{-1}Y_N,\,
h_n\right\rangle\right|\right]
\xrightarrow{N\to\infty}
\mathbf{E}[|Z|]\int_0^1 \mathbf{E}\!\left[e^{-bZ}\right]
e^{b(1-b)+\frac{1}{2}(1-b)^2}\,\mathrm{d}b
=: C_0 < \infty,
\]
where $Z \sim \mathcal{N}(0,1)$. In particular,
$\mathbf{E}[|\langle \mathrm{D}(-L)^{-1}Y_N, h_n\rangle|] \sim \frac{C_0}{\sqrt{2N}}$
as $N\to\infty$, uniformly in $n \in I_N$. Summing over $n \in I_N$ and inserting
into \eqref{eq:Distance_LogNormal}, we conclude that
\[
\delta_N \leqslant C\left(\frac{1}{\sqrt{N}}
+ \frac{\sharp I_N}{\sqrt{N}}\right)
= C\,\frac{1 + \sharp I_N}{\sqrt{N}},
\]
so $\delta_N$ is of order $\sharp I_N/\sqrt{N}$.

\begin{remark}
	\begin{enumerate}
		\item {\rm One possible interpretation is the following. If $\sharp I_N = N^{\alpha}$ with
	$\alpha < 1/2$, then the block of observations $(W(h_i),\, i \in I_N)$ carries
	asymptotically no information about $Y_N$ as $N \to \infty$: the conditional law of
	$Y_N$ given $(W(h_i),\, i \in I_N)$ is asymptotically the same as the unconditional
	law of $Y_N$.
}
\item 	{\rm One can also recover the order $\sharp I_N/\sqrt{N}$ by applying the bound introduced
	in \cite{TuZu2}, provided one swaps the roles of $Y_N$ and the vector
	$(W(h_i),\, i \in I_N)$, treating the Gaussian vector as the first component. This
	gives
	\begin{eqnarray*}
&&	d_{\mathrm{W}}\!\left(\bigl((W(h_i),\,i \in I_N),\,Y_N\bigr);\;
	\mathcal{N}(0, I_{\sharp I_N})\otimes\mathfrak{L}\right)
	\leqslant C\sum_{n \in I_N}
	\mathbf{E}\!\left[\left|\left\langle \mathrm{D}(-L)^{-1}W(h_n),\,
	\mathrm{D}Y_N\right\rangle\right|\right]\\
	&&
	= C\sum_{n \in I_N}
	\mathbf{E}\!\left[\left|\left\langle h_n,\,\mathrm{D}Y_N\right\rangle\right|\right]
	= \frac{C}{\sqrt{N}}\sum_{n \in I_N}\mathbf{E}[|W(h_n)|\,Y_N],
		\end{eqnarray*}
	where we used $\mathrm{D}(-L)^{-1}W(h_n) = h_n$ since $W(h_n)$ lies in the first
	Wiener chaos. Using the decomposition of $Y_N$ above, one checks that for every
	$1 \leqslant n \leqslant N$,
	\[
	\mathbf{E}[|W(h_n)|\,Y_N] \xrightarrow{N\to\infty} \mathbf{E}[|Z|]\,\mathbf{E}[e^Z],
	\]
	and one recovers the same order $\sharp I_N/\sqrt{N}$ for the distance.
	
	However, this role-swapping trick relies in an essential way on the fact that the
	second component $(W(h_i),\, i \in I_N)$ is itself Gaussian, which makes \cite{TuZu2}
	applicable after the swap. The natural and direct formulation of the problem ---
	bounding the distance to $\mathfrak{L}\otimes\mathcal{N}(0,I_{\sharp I_N})$ with the
	lognormal as the \emph{first} marginal target --- is handled directly by
	Theorem~\ref{tt1} of the present paper, and cannot be addressed by \cite{TuZu2} in
	this form. More importantly, if the second component $(W(h_i))$ were replaced by a
	non-Gaussian functional, the role-swapping would no longer be available and
	\cite{TuZu2} would be inapplicable altogether, whereas Theorem~\ref{tt1} would remain
	valid without any modification.
}
\end{enumerate}
\end{remark}

\section{Tools from Malliavin calculus}\label{sec:tools}

Here we describe the elements from stochastic analysis that we will need in the paper.
Let $\mathcal{H}$ be a real separable Hilbert space and let
$(W(h), h \in \mathcal{H})$ be an isonormal Gaussian process on a probability space
$(\Omega, \mathcal{F}, \mathbf{P})$, that is, a centered Gaussian family of random
variables such that $\mathbf{E}[W(\varphi)W(\psi)] = \langle\varphi,\psi\rangle_{\mathcal{H}}$
for all $\varphi, \psi \in \mathcal{H}$.

Denote by $I_n$ the multiple stochastic integral of order $n$ with respect to $W$
(see \cite{N}). The mapping $I_n$ is an isometry between the Hilbert space
$\mathcal{H}^{\odot n}$ (symmetric tensor product) equipped with the scaled norm
$\frac{1}{\sqrt{n!}}\|\cdot\|_{\mathcal{H}^{\otimes n}}$ and the Wiener chaos of
order $n$, defined as the closed linear span of the random variables $H_n(W(h))$
with $h \in \mathcal{H}$, $\|h\|_{\mathcal{H}} = 1$, where $H_n$ is the Hermite
polynomial of degree $n \in \mathbb{N}$ defined by
\begin{equation*}
	H_n(x) = \frac{(-1)^n}{n!}\exp\!\left(\frac{x^2}{2}\right)
	\frac{\mathrm{d}^n}{\mathrm{d}x^n}\exp\!\left(-\frac{x^2}{2}\right),
	\qquad x \in \mathbb{R}.
\end{equation*}
Note that this normalization includes the factor $1/n!$ and thus differs from some
references (e.g.\ \cite{N}, where the convention is $H_n^{\mathrm{NP}}(x) = n!\,H_n(x)$).
In particular, $\mathbf{E}[I_n(f)] = 0$ for all $f \in \mathcal{H}^{\otimes n}$ with
$n \geqslant 1$ (see e.g.\ Lemma~1.1.1 in \cite{N}). Since $I_n(f) = I_n(\tilde{f})$,
where $\tilde{f}$ denotes the symmetrization of $f$,
\begin{equation*}
	\tilde{f}(x_1,\ldots,x_n)
	= \frac{1}{n!}\sum_{\sigma \in \mathcal{S}_n}
	f(x_{\sigma(1)},\ldots,x_{\sigma(n)}),
\end{equation*}
the isometry of multiple integrals can be written as follows: for $m, n$ positive
integers,
\begin{eqnarray}
	\mathbf{E}\bigl[I_n(f)\,I_m(g)\bigr]
	&=& n!\,\langle\tilde{f},\tilde{g}\rangle_{\mathcal{H}^{\otimes n}}
	\quad \text{if } m = n, \nonumber \\
	\mathbf{E}\bigl[I_n(f)\,I_m(g)\bigr]
	&=& 0
	\quad \text{if } m \neq n. \label{iso}
\end{eqnarray}

\medskip

Any square-integrable random variable $F$ measurable with respect to the
$\sigma$-algebra generated by $W$ can be expanded into an orthogonal sum of multiple
stochastic integrals
\begin{equation}\label{sum1}
	F = \sum_{n=0}^\infty I_n(f_n),
\end{equation}
where $f_n \in \mathcal{H}^{\odot n}$ are uniquely determined symmetric kernels and
$I_0(f_0) = \mathbf{E}[F]$.

\medskip

Let $L$ be the Ornstein--Uhlenbeck operator defined by
\begin{equation*}
	LF = -\sum_{n \geqslant 1} n\,I_n(f_n)
\end{equation*}
for $F$ given by \eqref{sum1}, on the domain
$\mathrm{Dom}(L) = \bigl\{F \in L^2(\Omega) :
\sum_{n=1}^\infty n^2\,n!\,\|f_n\|^2_{\mathcal{H}^{\otimes n}} < \infty\bigr\}$.
Its pseudo-inverse $(-L)^{-1}$ is defined by $(-L)^{-1}F = \sum_{n=1}^\infty
\frac{1}{n}I_n(f_n)$ for $F$ with $\mathbf{E}[F] = 0$.

\medskip

For $p > 1$ and $\alpha \in \mathbb{R}$, the Sobolev--Watanabe space
$\mathbb{D}^{\alpha,p}$ is the closure of the set of polynomial random variables with
respect to the norm
\begin{equation*}
	\|F\|_{\alpha,p}
	= \|(\mathrm{Id} - L)^{\alpha/2}F\|_{L^p(\Omega)},
\end{equation*}
where $\mathrm{Id}$ denotes the identity operator (we use $\mathrm{Id}$ to avoid
confusion with the multiple stochastic integral $I_n$). We denote by $\mathrm{D}$
the Malliavin derivative operator, which acts on smooth random variables of the form
$F = g(W(h_1),\ldots,W(h_n))$ (with $g$ smooth with compact support and
$h_i \in \mathcal{H}$) by
\begin{equation*}
	\mathrm{D}F
	= \sum_{i=1}^n \frac{\partial g}{\partial x_i}(W(h_1),\ldots,W(h_n))\,h_i.
\end{equation*}
The operator $\mathrm{D}$ is continuous from $\mathbb{D}^{\alpha,p}$ into
$\mathbb{D}^{\alpha-1,p}(\mathcal{H})$. Its adjoint is the divergence operator
(Skorokhod integral), denoted by $\delta$, which acts from
$\mathbb{D}^{\alpha-1,p}(\mathcal{H})$ into $\mathbb{D}^{\alpha,p}$. The following
duality relation holds for every $F \in \mathbb{D}^{1,2}$ and
$u \in \mathbb{D}^{1,2}(\mathcal{H}) \subset \mathrm{Dom}(\delta)$:
\begin{equation}\label{dua}
	\mathbf{E}\bigl[F\,\delta(u)\bigr]
	= \mathbf{E}\bigl[\langle \mathrm{D}F,\, u\rangle_{\mathcal{H}}\bigr].
\end{equation}
A key identity used in the proof of Theorem~\ref{tt1} is the following: for any
$F \in \mathbb{D}^{1,2}$,
\begin{equation}\label{key}
	\delta\bigl(\mathrm{D}(-L)^{-1}(F - \mathbf{E}[F])\bigr) = F - \mathbf{E}[F].
\end{equation}
This follows from the relation $\delta \circ \mathrm{D} = -L$ on $\mathbb{D}^{1,2}$
and $-L(-L)^{-1}G = G$ for $G$ with $\mathbf{E}[G] = 0$.

\medskip

The product formula for multiple stochastic integrals is given for
$f \in \mathcal{H}^{\odot n}$ and $g \in \mathcal{H}^{\odot m}$ by
\begin{equation}\label{prod}
	I_n(f)\,I_m(g)
	= \sum_{r=0}^{n \wedge m} r!\binom{n}{r}\binom{m}{r}
	I_{m+n-2r}(f\,\widetilde{\otimes}_r\, g),
\end{equation}
where $f\,\widetilde{\otimes}_r\, g$ denotes the symmetrization of the $r$-contraction
$f \otimes_r g$ (see e.g.\ Section~1.1.2 in \cite{N}). The contraction is defined when
$\mathcal{H} = L^2(T, \mathbb{B}, \nu)$, with $\nu$ a non-atomic sigma-finite measure,
by
\begin{eqnarray}
	&&\label{contra}
	(f \otimes_r g)(t_1,\ldots,t_{n+m-2r}) \\
	&=& \int_{T^r} f(u_1,\ldots,u_r,t_1,\ldots,t_{n-r})\,
	g(u_1,\ldots,u_r,t_{n-r+1},\ldots,t_{n+m-2r})\,
	\nu(\mathrm{d}u_1)\cdots\nu(\mathrm{d}u_r), \nonumber
\end{eqnarray}
for $r = 1,\ldots,n \wedge m$, and $f \otimes_0 g = f \otimes g$ (tensor product).
Under this assumption $\mathcal{H} = L^2(T,\mathbb{B},\nu)$, we have
$f \otimes_r g \in \mathcal{H}^{\otimes n+m-2r} = L^2(T^{n+m-2r}, \nu^{\otimes(n+m-2r)})$.
In general, $f \otimes_r g$ is not symmetric and we denote its symmetrization by
$f\,\widetilde{\otimes}_r\, g$.

\end{document}